\documentclass{birkjour}
\usepackage{quaternions}

\usepackage{latexsym}
\usepackage{amsmath}
\usepackage{amsfonts}
\usepackage{amssymb}
\usepackage{graphicx}
\usepackage{amsthm}
\usepackage{amsmath}
\usepackage{url}

\newcommand{\ed}{\end{document}}
\usepackage[noadjust]{cite}
\DeclareMathOperator{\mymod}{\mathrm{mod}}

\usepackage{maple2eACA} 
\usepackage{mapletab}
\newenvironment{smallmaplegroup}%
{\small \begin{maplegroup}}{\end{maplegroup} \vskip-0.75ex}
\def\MapleLatexSize{\small} 
\newcommand{\mytab}{\phantom{>>\;}}

 \newtheorem{thm}{Theorem}[section]
 
 \newtheorem{lem}[thm]{Lemma}
 
 \theoremstyle{definition}
 
 \theoremstyle{remark}
 
 \newtheorem*{ex}{Example}
 \numberwithin{equation}{section}

\newcommand{\HQ}{\mathbb{H}}
\newcommand{\K}{\mathbb{K}}
\newcommand{\A}{\mathcal{A}}
\newcommand{\M}{\mathcal{M}}
\newcommand{\Aut}{\mathrm{Aut}}
\newcommand{\Inn}{\mathrm{Inn}}
\newcommand{\G}{\mathrm{G}}
\newcommand{\Cent}{\mathrm{Cent}}
\newcommand{\Zof}{\mathrm{Z}}
\newcommand{\Scal}{\mathrm{Scal}}
\newcommand{\Spec}{\mathrm{Spec}}
\newcommand{\unitm}{\mathbf{1}} 
\newcommand{\zerom}{\mathbf{0}} 
\newcommand{\trace}[1]{\mathrm{tr}(#1)}

\newcommand{\cl}{C \kern -0.1em \ell}     
\newcommand{\be}{\begin{equation}}
\newcommand{\ee}{\end{equation}}

\begin{document}
\title{Square Roots of $-1$ in Real Clifford Algebras}
\author[E. Hitzer]{Eckhard Hitzer}
\address{%
Department of Applied Physics,\\ 
University of Fukui,\\
Japan}
\email{hitzer@mech.fukui-u.ac.jp}

\author[J. Helmstetter]{Jacques Helmstetter}
\address{%
Univesit\'{e} Grenoble I,\\ 
Institut Fourier (Math\'{e}matiques),\\
B.P. 74, F-38402 Saint-Martin d'H\`{e}res,\\ 
France}
\email{Jacques.Helmstetter@ujf-grenoble.fr}

\author[R.~Ab\l amowicz]{Rafa\l \ Ab\l amowicz}
\address{%
Department of Mathematics, Box 5054,\\
Tennessee Technological University,\\
Cookeville, TN 38505, USA}
\email{rablamowicz@tntech.edu}

%



\date{February 17, 2012}

\begin{abstract}
It is well known that Clifford (geometric) algebra offers a geometric interpretation for square roots of $-1$ in the form of blades that square to minus $1$. This extends to a geometric interpretation of quaternions as the side face bivectors of a unit cube. Systematic research has been done~\cite{SJS:Biqroots} on the biquaternion roots of $-1$, abandoning the restriction to blades. Biquaternions are isomorphic to the Clifford (geometric) algebra $\cl(3,0)$ of $\R^3$. Further research on general algebras $\cl(p,q)$ has explicitly derived the geometric roots of $-1$ for  $p+q \leq 4$~\cite{HA:GeoRoots-1}. The current research abandons this dimension limit and uses the Clifford algebra to matrix algebra isomorphisms in order to algebraically characterize the continuous manifolds of square roots of $-1$ found in the different types of Clifford algebras, depending on the type of associated ring ($\R$, $\H$, $\R^2$, $\HQ^2$, or $\C$). At the end of the paper explicit computer generated tables of representative square roots of $-1$ are given for all Clifford algebras with $n=5,7$, and $s=3 \, (\mymod 4)$ with the associated ring $\C$. This includes, e.g., $\cl(0,5)$ important in Clifford analysis, and $\cl(4,1)$ which in applications is at the foundation of conformal geometric algebra. All these roots of $-1$ are immediately useful in the construction of new types of geometric Clifford Fourier transformations.
\end{abstract}


\maketitle

\section{Introduction}
The young London Goldsmid professor of applied mathematics W.~K.~Clifford created his \textit{geometric algebras}\footnote{%
In his original publication~\cite{WKC:AppGrass} Clifford first used the term \textit{geometric algebras}. Subsequently in mathematics the new term \textit{Clifford algebras}~\cite{PL:CAaSpin} has become the proper mathematical term. For emphasizing the \textit{geometric} nature of the algebra, some researchers continue~\cite{DH:NF1,HS:CAtoGC,GS:ncHCFT} to use the original term geometric algebra(s).} in 1878 inspired by the works of Hamilton on quaternions and by Grassmann's exterior algebra. Grassmann invented the antisymmetric outer product of vectors, that regards the oriented parallelogram area spanned by two vectors as a new type of number, commonly called bivector. The bivector represents its own plane, because outer products with vectors in the plane vanish. In three dimensions the outer product of three linearly independent vectors defines a so-called trivector with the magnitude of the volume of the parallelepiped spanned by the vectors. Its orientation (sign) depends on the handedness of the three vectors. 

In the Clifford algebra~\cite{DH:NF1} of $\R^3$ the three bivector side faces of a unit cube $\{{e}_1{e}_2, {e}_2{e}_3, {e}_3{e}_1\}$ oriented along the three coordinate directions $\{{e}_1, {e}_2, {e}_3\}$ correspond to the three quaternion units $\i$, $\j$, and $\k$. Like quaternions, these three bivectors square to minus one and generate the rotations in their respective planes. 

Beyond that Clifford algebra allows to extend complex numbers to higher dimensions~\cite{HS:CAtoGC,BDS:CA} and systematically generalize our knowledge of complex numbers, holomorphic functions and quaternions into the realm of Clifford analysis. It has found rich applications in symbolic computation, physics, robotics, computer graphics, etc.~\cite{GS:ncHCFT,TB:thesis,MF:thesis,HL:IAGR,DL:GAinPrac}. Since bivectors and trivectors in the Clifford algebras of Euclidean vector spaces square to minus one, we can use them to create new geometric kernels for Fourier transformations. This leads to a large variety of new Fourier transformations, which all deserve to be studied in their own right~\cite{LMQ:CAFT94,AM:CAFT96,TQ:PWT,ES:CFTonVF,HM:CFTUP,MHAV:WQFT,HM:CFaUP,HM:CFToMVF,EH:QFTgen,MHA:2DCliffWinFT,SBS:FastCmQFT,GS:ncHCFT,EH:OPS-QFT}. 

In our current research we will treat square roots of $-1$ in Clifford algebras $\cl(p,q)$ of both Euclidean (positive definite metric) and non-Euclidean (indefinite metric) non-degenerate vector spaces, $\R^{n}=\R^{n,0}$ and $\R^{p,q}$, respectively. We know from Einstein's special theory of relativity that non-Euclidean vector spaces are of fundamental importance in nature~\cite{DH:STA}. They are further, e.g., used in computer vision and robotics~\cite{DL:GAinPrac} and for general algebraic solutions to contact problems \cite{HL:IAGR}. Therefore this chapter is about characterizing square roots of $-1$ in all Clifford algebras $\cl(p,q)$, extending previous limited research on $\cl(3,0)$ in~\cite{SJS:Biqroots} and $\cl(p,q), n=p+q\leq 4$ in~\cite{HA:GeoRoots-1}. The manifolds of square roots of $-1$ in $\cl(p,q)$, $n=p+q=2$, compare Table 1 of~\cite{HA:GeoRoots-1}, are visualized in Fig. \ref{fg:Cln=2}.

\begin{figure}
  \begin{center}      
    \includegraphics[scale=0.35]{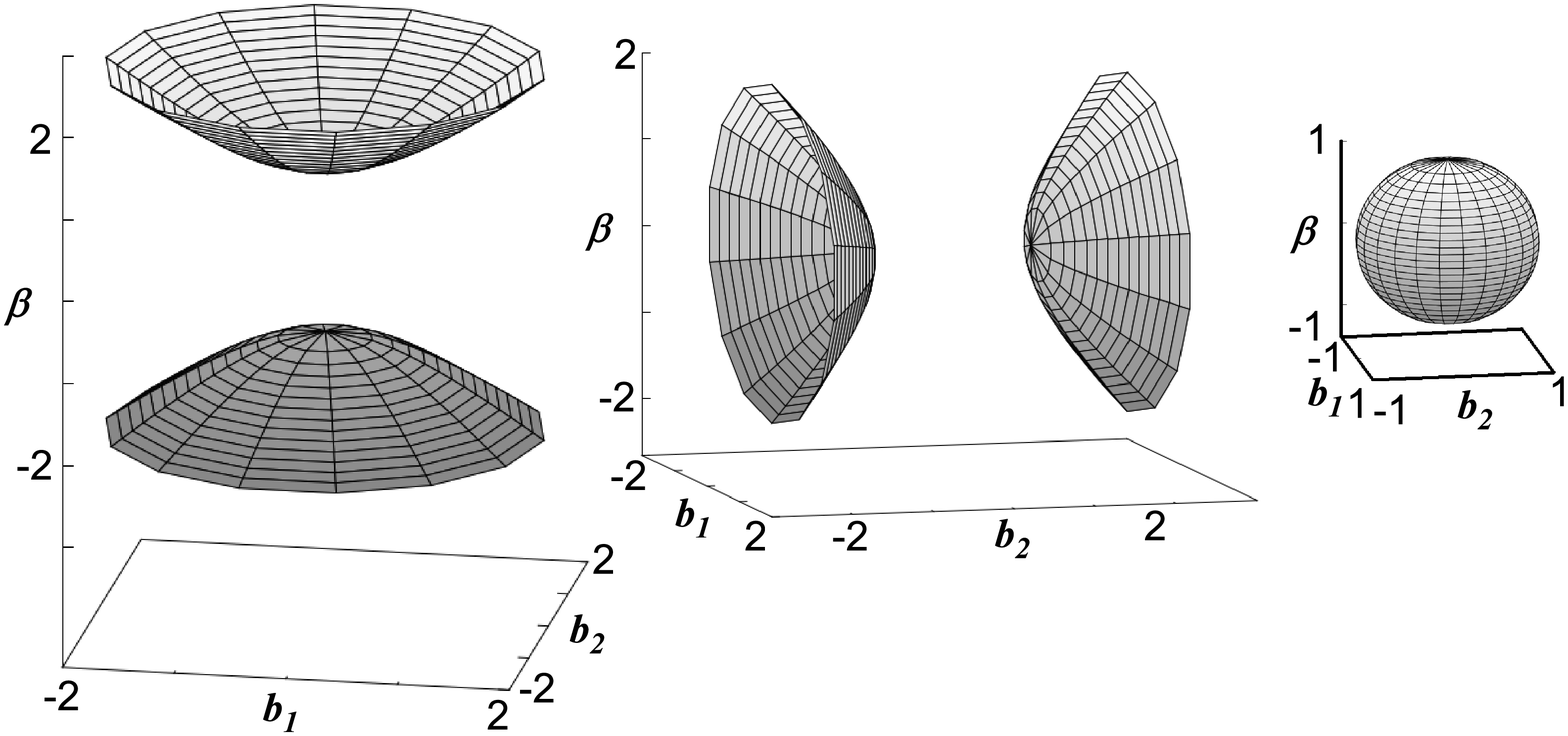}
    \caption{Manifolds of square roots $f$ of $-1$ in $\cl(2,0)$ (left), $\cl(1,1)$ (center), and $\cl(0,2)\cong\H$ (right). The square roots are $f=\alpha + b_1 e_1+b_2e_2+\beta e_{12},$ with $\alpha, b_1, b_2, \beta \in \R$, $\alpha=0$, and $\beta^2=b_1^2e_2^2+b_2^2e_1^2+e_1^2e_2^2$.}  
    \label{fg:Cln=2}
  \end{center}
\end{figure}

First, we introduce necessary background knowledge of Clifford algebras and matrix ring isomorphisms and explain in more detail how we will characterize and classify the square roots of $-1$ in Clifford algebras in Section~\ref{sc:intro}. Next, we treat section by section (in Sections \ref{sc:M2dR} to \ref{sc:M2dC}) the square roots of $-1$ in Clifford algebras which are isomorphic to matrix algebras  with associated rings $\R$, $\H$, $\R^2$, $\HQ^2$, and $\C$, respectively. The term \textit{associated} means that the isomorphic matrices will only have matrix elements from the associated ring. The square roots of $-1$ in Section \ref{sc:M2dC} with associated ring $\C$ are of particular interest, because of the existence of classes of \textit{exceptional} square roots of $-1$, which all include a nontrivial term in the central element of the respective algebra different from the identity. Section \ref{sc:M2dC} therefore includes a detailed discussion of all classes of square roots of 
$-1$ in the algebras $\cl(4,1)$, the isomorphic $\cl(0,5)$, and in $\cl(7,0)$. Finally, we add appendix \ref{AppendA} with tables of square roots of $-1$ for all Clifford algebras with $n=5,7$, and 
$s=3 \, (\mymod 4)$. The square roots of $-1$ in Section \ref{sc:M2dC} and in Appendix~\ref{AppendA} were all computed with the Maple package CLIFFORD~\cite{AF:CLIFFORD}, as explained 
in Appendix~\ref{AppendB}.

\section{Background and problem formulation} 
\label{sc:intro}
Let $\cl(p, q)$ be the algebra (associative with unit $1$) generated over $\R$ by $p + q$ elements $e_k$ (with $k = 1, 2,\ldots, p + q)$ with the relations $e^2_k = 1$ if $k \leq p$, $e^2_k = -1$ if 
$k > p$ and $e_he_k + e_ke_h = 0$ whenever $h \neq k$, see \cite{PL:CAaSpin}. We set the vector space dimension $n = p + q$ and the signature $s = p - q$. This algebra has dimension $2^n$, and its even subalgebra $\cl_0(p, q)$ has dimension  $2^{n-1}$ (if $n > 0$). We are concerned with square roots of $-1$ contained in $\cl(p, q)$ or $\cl_0(p, q)$. If the dimension of $\cl(p, q)$ or, 
$\cl_0(p, q)$ is $\leq 2$, it is isomorphic to $\R\cong \cl(0,0)$, $\R^2\cong \cl(1,0)$, or $\C\cong \cl(0,1)$, and it is clear that there is no square root of $-1$ in $\R$ and 
$\R^2 = \R \times \R$, and that there are two squares roots $i$ and $-i$ in $\C$. Therefore we only consider algebras of dimension $\geq 4$. Square roots of $-1$ have been computed explicitly in \cite{SJS:Biqroots} for $\cl(3,0)$, and in \cite{HA:GeoRoots-1} for algebras of dimensions $2^n\leq 16$.

An algebra $\cl(p, q)$ or $\cl_0(p, q)$ of dimension $\geq 4$ is isomorphic to one of the five matrix algebras: $\M(2d, \R)$, $\M(d,\HQ)$, $\M(2d, \R^2)$, $\M(d,\HQ^2)$ or $\M(2d, \C)$. The integer $d$ depends on $n$. According to the parity of $n$, it is either $2^{(n-2)/2}$ or $2^{(n-3)/2}$ for $\cl(p, q)$, and, either $2^{(n-4)/2}$ or $2^{(n-3)/2}$ for $\cl_0(p, q)$. The associated ring 
(either $\R$, $\HQ$, $\R^2$, $\HQ^2$, or $\C$) depends on $s$ in this way\footnote{Compare chapter 16 on \textit{matrix representations and periodicity of 8}, as well as Table 1 on p. 217 of \cite{PL:CAaSpin}.}:
\begin{table}[h!]
\begin{center}
\begin{tabular}{r|c|c|c|c|c|c|c|c|}
  $s$ mod 8 & 0 & 1 & 2 & 3 & 4 & 5 & 6 & 7 \\
  \hline 
  \rule{0mm}{4.5mm}%
  associated ring for $\cl(p,q)$ 
  & $\R$ & $\R^2$ & $\R$ & $\C$ & $\HQ$ & $\HQ^2$ & $\HQ$ & $\C$ \\
  \hline 
  \rule{0mm}{4.5mm}%
  associated ring for $\cl_0(p,q)$ 
  & $\R^2$ & $\R$ & $\C$ & $\HQ$ & $\HQ^2$ & $\HQ$ & $\C$ & $\R$ \\
  \hline
\end{tabular}
\end{center}
\end{table}

\noindent
Therefore we shall answer this question: What can we say about the square roots of $-1$ in an algebra $\A$ that is isomorphic to $\M(2d, \R)$, $\M(d,\HQ)$, $\M(2d, \R^2)$, $\M(d,\HQ^2)$, or, $\M(2d, \C)$? They constitute an algebraic submanifold in $\A$; how many connected components\footnote{Two points are in the same connected component of a manifold, if they can be joined by a continuous path inside the manifold under consideration. (This applies to all topological spaces satisfying the property that each neighborhood of any point contains a neighborhood in which every pair of points can always be joined by a continuous path.)}
(for the usual topology) does it contain? Which are their dimensions? This submanifold is invariant by the action of the \textit{group} $\Inn(\A)$ of \textit{inner automorphisms}\footnote{An inner automorphism $f$ of $\A$ is defined as $f:\A \rightarrow \A, f(x) = a^{-1}xa, \forall x \in \A$, with given fixed $a\in \A$. The composition of two inner automorphisms $g(f(x))=b^{-1}a^{-1}xab = (ab)^{-1}x(ab)$ is again an inner automorphism. With this operation the inner automorphisms form the group $\Inn(\A)$, compare \cite{WP:innaut}.} of $\A$, i.e. for every $ r\in \A, r^2 = -1 \Rightarrow f(r)^2=-1\,\,\, \forall f\in \Inn(\A)$. The orbits of $\Inn(\A)$ are called conjugacy classes\footnote{The conjugacy class (similarity class) of a given $r\in \A, r^2 = -1$ is $\{f(r): f\in \Inn(\A) \}$, compare \cite{WP:conjclass}. Conjugation is transitive, because the composition of inner automorphisms is again an inner automorphism.}; how many conjugacy classes are there in this submanifold? If the associated ring is $\R^2$ or $\HQ^2$ or $\C$, the group $\Aut(\A)$ of all automorphisms of $\A$ is larger than $\Inn(\A)$, and the action of $\Aut(\A)$ in this submanifold shall also be described.

We recall some properties of $\A$ that do not depend on the associated ring. The group $\Inn(\A)$ contains as many connected components as the \textit{group} $\G(\A)$ of \textit{invertible elements} in $\A$. We recall that this assertion is true for $\M(2d, \R)$ but not for $\M(2d+1, \R)$ which is not one of the relevant matrix algebras. If $f$ is an element of $\A$, let $\Cent(f)$ be the centralizer of $f$, that is, the subalgebra of all $g \in \A$ such that $fg = gf$. The conjugacy class of $f$ contains as many connected components\footnote{According to the general theory of groups acting on sets, the conjugacy class (as a topological space) of a square root $f$ of $-1$ is isomorphic to the \textit{quotient} of $\G(\A)$ and $\Cent(f)$ (the subgroup of stability of $f$). Quotient means here the set of left handed classes modulo the subgroup. If the subgroup is contained in the neutral connected component of $\G(\A)$, then the number of connected components is the same in the quotient as in $\G(\A)$. See also \cite{CC:LieGroups}.} as $\G(\A)$ if (and only if) $\Cent(f)\bigcap\G(\A)$ is contained in the neutral\footnote{\textit{Neutral} means to be connected to the identity element of $\A$.} connected component of $\G(\A)$, and the dimension of its conjugacy class is 
\be 
\dim(\A) - \dim(\Cent(f)).
\label{eq:dimconjclass}
\ee 
Note that for invertible $g \in \Cent(f)$ we have $g^{-1}fg = f$.

Besides, let $\Zof(\A)$ be the center of $\A$, and let $[\A,\A]$ be the subspace spanned by all $[f, g] = fg-gf$. In all cases $\A$ is the direct sum of $\Zof(\A)$ and $[\A,\A]$. For 
example,\footnote{A matrix algebra based proof is e.g., given in \cite{uwp:center}.} $\Zof(\M(2d,\R)) = \{a\unitm\mid a\in \R\}$ and $\Zof(\M(2d,\C)) = \{c\unitm\mid c\in \C\}$. If the associated ring is $\R$ or $\HQ$ (that is for even $n$), then $\Zof(\A)$ is canonically isomorphic to $\R$, and from the projection $\A \rightarrow \Zof(\A)$ we derive a linear form $\Scal : \A \rightarrow \R$. When the associated ring\footnote{This is the case for $n$ (and  $s$) odd. Then the pseudoscalar $\omega \in \cl(p,q)$ is also in $\Zof(\cl(p,q))$.} is $\R^2$ or $\HQ^2$ or $\C$, then $\Zof(\A)$ is spanned by $\unitm$ (the unit matrix\footnote{The number $1$ denotes the unit of the Clifford algebra $\A$, whereas the bold face $\unitm$ denotes the unit of the isomorphic matrix algebra $\M$.}) and some element $\omega$ such that $\omega^2 = \pm \unitm$. Thus, we get two linear forms $\Scal$ and $\Spec$ such that $\Scal(f)\unitm+\Spec(f)\omega$ is the projection of $f$ in $\Zof(\A)$ for every $f \in \A$. Instead of $\omega$ we may use $-\omega$ and replace $\Spec$ with $-\Spec$. The following assertion holds for every $f \in \A$: The trace of each multiplication\footnote{These multiplications are bilinear over the center of $\A$.} $g \mapsto fg$ or $g \mapsto gf$ is equal to the product 
\be 
\trace{f} = \dim(\A)\, \Scal(f).
\label{eq:trace}
\ee 
The word ``trace" (when nothing more is specified) means a matrix trace in $\R$, which is the sum of its diagonal elements. For example, the matrix $M\in \mathcal{M}(2d,\R)$ with elements 
$m_{kl} \in \R, 1\leq k,l\leq 2d$ has the trace $\trace{M} = \sum_{k=1}^{2d}m_{kk}$~\cite{HJ:MatAn}. 

We shall prove that in all cases $\Scal(f) = 0$ for every square root of $-1$ in $\A$. Then, we may distinguish \textit{ordinary} square roots of $-1$, and \textit{exceptional} ones. In all cases the ordinary \label{pg:ordinary} square roots of $-1$ constitute a unique\footnote{Let $\A$ be an algebra $\M(m,\K)$ where $\K$ is a division ring. Thus two elements $f$ and $g$ of $\A$ induce $\K$-linear endomorphisms $f'$ and $g'$ on $\K^m$; if $\K$ is not commutative, $\K$ operates on $\K^m$ on the right side. The matrices $f$ and $g$ are conjugate (or similar) if and only if there are two $\K$-bases $B_1$ and $B_2$ of $\K^m$ such that $f'$ operates on $B_1$ in the same way as $g'$ operates on $B_2$. This theorem allows us to recognize that in all cases but the last one (with exceptional square roots of $-\unitm$), two square roots of $-\unitm$ are always conjugate.} conjugacy class of dimension $\dim(\A)/2$ which has as many connected components as $\G(\A)$, and they satisfy the equality $\Spec(f) = 0$ if the associated ring is $\R^2$ or $\HQ^2$ or $\C$. The exceptional square roots of $-1$ only exist\footnote{The pseudoscalars of Clifford algebras whose isomorphic matrix algebra has ring  $\R^2$ or $\HQ^2$ square to $\omega^2=+1$. } if $\A \cong \M(2d,\C)$. In $\M(2d,\C)$ there are $2d$ conjugacy classes of exceptional square roots of $-1$, each one characterized by an equality $\Spec(f) = k/d$ with $\pm k \in \{1, 2, \ldots, d\}$ [see Section \ref{sc:M2dC}], and their dimensions are $< \dim(\A)/2$ [see eqn. \eqref{eq:M2dCRdim}]. For instance, $\omega$ (mentioned above) and  $-\omega$  are central square roots of $-1$ in $\M(2d, \C)$ which constitute two conjugacy classes of dimension~$0$. Obviously, $\Spec(\omega) = 1$.

For symbolic computer algebra systems (CAS), like MAPLE, there exist Clifford algebra packages, e.g., CLIFFORD \cite{AF:CLIFFORD}, which can compute idempotents \cite{AFPR:idem} and square roots of 
$-1$. This will be of especial interest for the exceptional square roots of $-1$ in $\M(2d, \C)$. 

Regarding a square root $r$ of $-1$, a Clifford algebra is the direct sum of the subspaces $\mathrm{Cent}(r)$ (all elements that commute with $r$) and the skew-centralizer $\mathrm{SCent}(r)$ (all elements that anticommute with $r$). Every Clifford algebra multivector has a unique split by this Lemma.
\begin{lem}
Every multivector $A \in \cl(p,q)$ has, with respect to a square root $r\in \cl(p,q)$ of~$-1$, i.e., $r^{-1}=-r,$ the unique decomposition
\be 
A_{\pm} = \frac{1}{2}(A \pm r^{-1}Ar), \quad A = A_++A_-, \quad A_+r = r A_+, \quad A_-r = -rA_-. 
\ee
\end{lem}
\begin{proof}
For $A \in \cl(p,q)$ and a square root $r\in \cl(p,q)$ of $-1$, we compute
\begin{align*} 
A_{\pm}r = \frac{1}{2}(A \pm r^{-1}Ar)r &= \frac{1}{2}(Ar \pm r^{-1}A(-1))\stackrel{r^{-1}=-r}{=}\frac{1}{2}(rr^{-1}Ar \pm rA)\\ 
                                        &= \pm r \frac{1}{2}(A \pm r^{-1}Ar).
\end{align*}
\renewcommand{\qedsymbol}{} 
\end{proof}
For example, in Clifford algebras $\cl(n,0)$ \cite{HM:CFToMVF} of dimensions $n=2 \, \mymod 4$, $\mathrm{Cent}(r)$ is the even subalgebra $\cl_{0}(n,0)$ for the unit pseudoscalar $r$, and the subspace $\cl_1(n,0)$ spanned by all $k$-vectors of odd degree $k$, is  $\mathrm{SCent}(r)$. The most interesting case is $\M(2d, \C)$, where a whole range of conjugacy classes becomes available. These results will therefore be particularly relevant for constructing \textit{Clifford Fourier transformations} using the square roots of $-1$. 

\section{Square roots of $-1$ in $\M(2\lowercase{d}, \R)$}
\label{sc:M2dR}
Here $\A =\M(2d, \R)$, whence $\dim(\A) = (2d)^2 = 4d^2$. The group $\G(\A)$ has \textit{two} connected components determined by the inequalities $\det(g) > 0$ and $\det(g) < 0$.

For the case $d=1$ we have, e.g., the algebra $\cl(2,0)$ isomorphic to $\M(2, \R)$. The basis $\{1, e_1, e_2, e_{12}\}$ of $\cl(2,0)$ is mapped to 
$$
\left\{
\begin{pmatrix}
1 & 0 \\
0 & 1
\end{pmatrix},
\begin{pmatrix}
0 & 1 \\
1 & 0
\end{pmatrix},
\begin{pmatrix}
1 & 0 \\
0 & -1
\end{pmatrix},
\begin{pmatrix}
0 & -1 \\
1 & 0
\end{pmatrix}
\right\}.
$$ 
The general element $\alpha + b_1 e_1 + b_2 e_2 +\beta e_{12} \in \cl(2,0)$ is thus mapped to 
\be 
\begin{pmatrix}
\alpha+b_2 & -\beta+b_1 \\
\beta+b_1 & \alpha-b_2
\end{pmatrix}
\label{eq:M2Rmat}
\ee 
in $\M(2, \R)$. Every element $f$ of $\A =\M(2d, \R)$ is treated as an $\R$-linear endomorphism of $V = \R^{2d}$. Thus, its scalar component and its trace \eqref{eq:trace} are related as follows: $\trace{f} = 2d \Scal(f)$. If $f$ is a square root of $-\unitm$, it turns $V$ into a vector space over $\C$ (if the complex number $i$ operates like $f$ on~$V$). If $(e_1, e_2, \ldots , e_d)$ is a 
$\C$-basis of $V$, then $(e_1, f(e_1), e_2, f(e_2), \ldots , e_d, f(e_d))$ is a $\R$-basis of $V$, and the $2d \times 2d$ matrix of $f$ in this basis is
\be
\mathrm{diag} \bigg(\underbrace{\left(\begin{matrix} 0 & -1\\1 & 0\end{matrix}\right),\ldots,\left(\begin{matrix} 0 & -1\\1 & 0\end{matrix}\right)}_{d}\bigg)
  \label{eq:M2dRimat}
\ee
Consequently all square roots of $-\unitm$ in $\A$ are conjugate. The centralizer of a square root~$f$ of~$-\unitm$ is the algebra of all $\C$-linear endomorphisms $g$ of $V$ (since $i$ operates like $f$ on $V$). Therefore, the $\C$-dimension of $\Cent(f)$ is $d^2$ and its 
$\R$-dimension is $2d^2$. Finally, the dimension \eqref{eq:dimconjclass} of the conjugacy class of $f$ is $\dim(\A)-\dim(\Cent(f))=4d^2 - 2d^2 = 2d^2 = \dim(\A)/2$. The two connected components of $\G(\A)$ are determined by the sign of the determinant. Because of the next lemma, the $\R$-determinant of every element of $\Cent(f)$ is $\geq 0$. Therefore, the intersection 
$\Cent(f)\bigcap \G(\A)$ is contained in the neutral connected component of $\G(\A)$ 
and, consequently, the conjugacy class of $f$ has two connected components like $\G(\A)$. Because of the next lemma, the $\R$-trace of $f$ vanishes (indeed its $\C$-trace is $di$, because $f$ is the multiplication by the scalar $i$: $f(v)=iv$ for all $v$) whence $\Scal(f) = 0$. This equality is corroborated by the matrix written above.

We conclude that the square roots of $-\unitm$ constitute one conjugacy class with two connected components of dimension $\dim(\A)/2$ contained in the hyperplane defined by the equation 
\be 
\Scal(f) = 0.
\label{eq:M2dRscal0}
\ee

Before stating the lemma that here is so helpful, we show what happens in the easiest case $d = 1$. The square roots of $-\unitm$ in $\M(2,\R)$ are the real matrices
\be
\begin{pmatrix}
a & c \\
b & -a
\end{pmatrix}
\mbox{ with }
\begin{pmatrix}
a & c \\
b & -a
\end{pmatrix}
\begin{pmatrix}
a & c \\
b & -a
\end{pmatrix}
=
(a^2 + bc) \,\unitm = -\unitm ;
\ee  
hence $a^2 + bc = -1$, a relation between $a, b, c$ which is equivalent to $(b - c)^2 = (b + c)^2 + 4a^2 + 4 \Rightarrow (b-c)^2\geq 4 \Rightarrow b-c \geq 2$ (one component) or $c-b \geq 2$ (second component). Thus, we recognize the two connected components of square roots of $-\unitm$: The inequality $b \geq c + 2$ holds in one connected component, and the inequality 
$c \geq b + 2$ in the other one, compare Fig.~\ref{fg:M2dR}.
\begin{figure}
  \begin{center}      
    \includegraphics[scale=0.5]{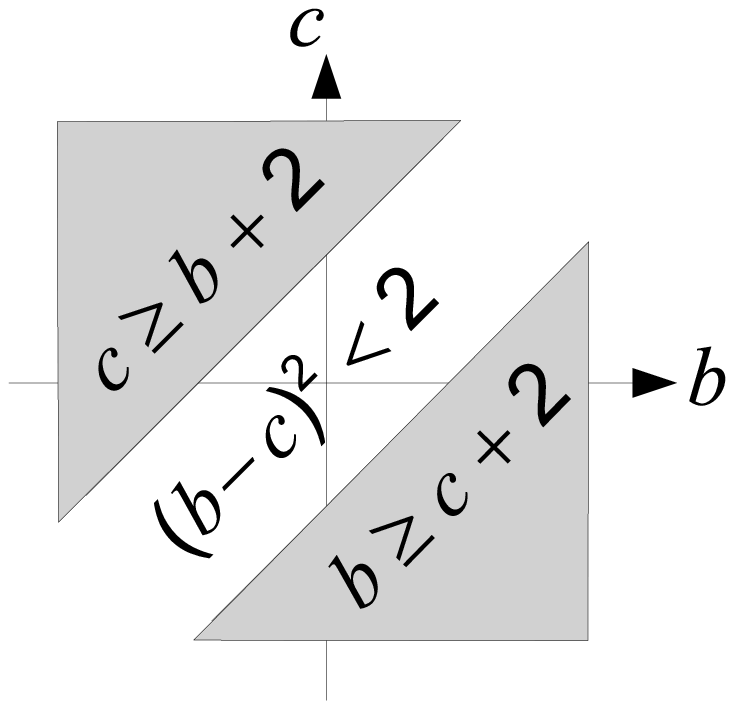}
    \caption{Two components of square roots of $-\unitm$ in $\M(2,\R)$}  
    \label{fg:M2dR}
  \end{center}
\end{figure}

In terms of $\cl(2,0)$ coefficients \eqref{eq:M2Rmat} with $b-c = \beta+b_1-(-\beta+b_1)=2\beta$, we get the two component conditions simply as 
\be 
\beta \geq 1  \quad \mbox{(one component)}, \qquad \beta \leq -1 \quad \mbox{ (second component)}.
\label{eq:M2Rcomps}
\ee 
Rotations ($\det(g)=1$) leave the pseudoscalar $\beta e_{12}$ invariant (and thus preserve the two connected components of square roots of $-\unitm$), but reflections ($\det(g')=-1$) change its sign 
$\beta e_{12}\rightarrow -\beta e_{12}$ (thus interchanging the two components). 

Because of the previous argument involving a complex structure on the real space $V$, we conversely consider the complex space $\C^d$ with its structure of vector space over $\R$. If 
$(e_1, e_2, \ldots , e_d)$ is a $\C$-basis of $\C^d$, then $(e_1, ie_1, e_2, ie_2, \ldots, e_d, ie_d)$ is a $\R$-basis. Let $g$ be a $\C$-linear endomorphism of $\C^d$ (i.e., a complex $d \times d$ matrix), let $\mathrm{tr}_{\C}(g)$ and $\det_{\C}(g)$ be the trace and determinant of $g$ in $\C$, and $\mathrm{tr}_{\R}(g)$ and $\det_{\R}(g)$ its trace and determinant for the real structure of $\C^d$.
\begin{ex}
For $d=1$ an endomorphism of $\C^1$ is given by a complex number $g=a+ib, \,a,b \in \R$. Its matrix representation is according to \eqref{eq:M2dRimat}
\be 
  \begin{pmatrix}
    a & -b \\
    b & a
  \end{pmatrix}
  \mbox{ with }
  \begin{pmatrix}
    a & -b \\
    b & a
  \end{pmatrix}^2
  =
  (a^2-b^2)
  \begin{pmatrix}
    1 & 0 \\
    0 & 1
  \end{pmatrix}
  +2ab
  \begin{pmatrix}
    0 & -1 \\
    1 & 0
  \end{pmatrix}.
\ee 
Then we have $\mathrm{tr}_{\C}(g)=a+ib$, 
$\mathrm{tr}_{\R}
  \begin{pmatrix}
    a & -b \\
    b & a
  \end{pmatrix}
=2a = 2\Re(\mathrm{tr}_{\C}(g))$
and $\det_{\C}(g) = a+ib$, 
$\det_{\R}
  \begin{pmatrix}
    a & -b \\
    b & a
  \end{pmatrix}
=a^2+b^2 = |\det_{\C}(g)|^2 \geq 0$.
\end{ex}
\begin{lem}
For every $\C$-linear endomorphism $g$ we can write $\mathrm{tr}_{\R}(g) = 2\Re(\mathrm{tr}_{\C}(g))$ and $\det_{\R}(g)$ $= | \det_{\C}(g) |^2 \geq 0$.
\end{lem}
\begin{proof} 
There is a $\C$-basis in which the $\C$-matrix of $g$ is triangular [then  $\det_{\C}(g)$ is the product of the entries of $g$ on the main diagonal]. We get the $\R$-matrix of $g$ in the derived $\R$-basis by replacing every entry $a+bi$ of the $\C$-matrix with the elementary matrix
$\begin{pmatrix}
  a & -b \\
  b & a
  \end{pmatrix}
$. 
The conclusion soon follows. The fact that the determinant of a block triangular matrix is the product of the determinants of the blocks on the main diagonal is used. 
\end{proof}

\section{Square roots of $-1$ in $\M(2\lowercase{d}, \R^2)$\label{sc:M2dR2}}

Here $\A = \M(2d, \R^2) = \M(2d,\R) \times \M(2d,\R)$, whence $\dim(\A) = 8d^2$. The group 
$\G(\A)$ has four\footnote{In general, the number of connected components of $\G(\A)$ is two if $\A=\M(m,\R)$, and one if $\A=\M(m,\C)$ or $\A=\M(m,\HQ)$, because in all cases every matrix can be joined by a continuous path to a diagonal matrix with entries $1$ or $-1$. When an algebra $\A$ is a direct product of two algebras $\mathcal{B}$ and $\mathcal{C}$, then $\G(\A)$ is the direct product of $\G(\mathcal{B})$ and $\G(\mathcal{C})$, and the number of connected components of $\G(\A)$ is the product of the numbers of connected components of $\G(\mathcal{B})$ and $\G(\mathcal{C})$. \label{fn:NoConComp}} connected components. Every element $(f, f') \in \A$ (with $f, f' \in \M(2d,\R)$) has a determinant in $\R^2$ which is obviously $(\det(f), \det(f'))$, and the four connected components of $\G(\A)$ are determined by the signs of the two components of $\det_{\R^2} (f, f')$.

The lowest dimensional example ($d=1$) is $\cl(2,1)$ isomorphic to $\M(2, \R^2)$. Here the pseudoscalar $\omega = e_{123}$ has square $\omega^2 = +1$. The center of the algebra is $\{1,\omega\}$ and includes the idempotents $\epsilon_{\pm}=(1{\pm}\omega)/2$, $\epsilon_{\pm}^2 = \epsilon_{\pm}$, $\epsilon_{+}\epsilon_{-}=\epsilon_{-}\epsilon_{+}=0$. The basis of the algebra can thus be written as 
$\{\epsilon_{+}, e_1\epsilon_{+}, e_2\epsilon_{+}, e_{12}\epsilon_{+}, \epsilon_{-}, e_1\epsilon_{-}, e_2\epsilon_{-}, e_{12}\epsilon_{-}\}$, where the first (and the last) four elements form a basis of the subalgebra $\cl(2,0)$ isomorphic to $\M(2,\R)$. In terms of matrices we have the identity matrix $(\unitm,\unitm)$ representing the scalar part, the idempotent matrices $(\unitm,0)$, 
$(0,\unitm)$, and the $\omega$ matrix $(\unitm,-\unitm)$, with~$\unitm$ the unit matrix of $\M(2,\R)$.

The square roots of $(-\unitm,-\unitm)$ in $\A$ are pairs of two square roots of $-\unitm$ in $\M(2d,\R)$. Consequently they constitute a unique conjugacy class with four connected components 
of dimension $4d^2 = \dim(\A)/2$. This number can be obtained in two ways. First, since every element $(f, f') \in \A$ (with $f, f' \in \M(2d,\R)$) has twice the dimension of the components $f \in \M(2d,\R)$ of Section \ref{sc:M2dR}, we get the component dimension $2\cdot 2d^2 = 4d^2$. Second, the centralizer $\Cent(f,f')$ has twice the dimension of $\Cent(f)$ of $\M(2d,\R)$, therefore 
$\dim(\A)-\Cent(f,f') = 8d^2 -4d^2 = 4d^2$. In the above example for $d=1$ the four components are characterized according to \eqref{eq:M2Rcomps} by the values of the coefficients of $\beta e_{12}\epsilon_{+}$ and $\beta{}' e_{12}\epsilon_{-}$ as
\begin{alignat}{2}
  c_1: \quad \beta &\geq 1,         &\beta' &\geq 1,  \nonumber \\
  c_2: \quad \beta &\geq 1,         &\beta' &\leq -1, \nonumber \\
  c_3: \quad \beta &\leq -1,        &\beta' &\geq 1,  \nonumber \\
  c_4: \quad \beta &\leq -1, \qquad &\beta' &\leq -1.
  \label{eq:M2dRcomps}
\end{alignat}
For every $(f,f') \in \A$ we can with \eqref{eq:trace} write $\trace{f} + \trace{f'} = 2d \Scal(f,f')$ and
\be 
\trace{f} - \trace{f'}  = 2d \Spec(f, f') \quad \mbox{if} \quad \omega = (\unitm,-\unitm) ;
\label{eq:M2dR2}
\ee
whence $\Scal(f, f') = \Spec(f, f') = 0$ if $(f, f')$ is a square root of $(-\unitm,-\unitm)$, compare \eqref{eq:M2dRscal0}.

The group $\Aut(\A)$ is larger than $\Inn(\A)$, because it contains the swap automorphism $(f,f')$ $\mapsto (f',f)$ which maps the central element $\omega$ to $-\omega$, and interchanges the two idempotents $\epsilon_+$ and $\epsilon_-$. The group $\Aut(\A)$ has eight connected components which permute the four connected components of the submanifold of square roots of $(-\unitm,-\unitm)$. The permutations induced by $\Inn(\A)$ are the permutations of the \textit{Klein group}. For example for $d=1$ of \eqref{eq:M2dRcomps} we get the following $\Inn(\M(2,\R^2))$ permutations
\begin{align}
  \det(g)>0, \quad \det(g')>0: &\quad \mbox{identity}, \nonumber \\
  \det(g)>0, \quad \det(g')<0: &\quad (c_1,c_2), (c_3,c_4), \nonumber \\
  \det(g)<0, \quad \det(g')>0: &\quad (c_1,c_3), (c_2,c_4), \nonumber \\
  \det(g)<0, \quad \det(g')<0: &\quad (c_1,c_4), (c_2,c_3). 
  \label{eq:M2dRKleingr}
\end{align}
Beside the identity permutation, $\Inn(\A)$ gives the three permutations that permute two elements and also the other two ones. 

The automorphisms outside $\Inn(\A)$ are 
\be 
(f,f') \mapsto (gf'g^{-1}, g'fg'^{-1}) \quad \mbox{for some} \quad (g, g') \in \G(\A).
\label{eq:M2dRoAut}
\ee 
If $\det(g)$ and $\det(g')$ have opposite signs, it is easy to realize that this automorphism induces a circular permutation on the four connected components of square roots of 
$(-\unitm,-\unitm)$: If $\det(g)$ and $\det(g')$ have the same sign, this automorphism leaves globally invariant two connected components, and permutes the other two ones. For example, 
for $d=1$ 
the automorphisms \eqref{eq:M2dRoAut}  outside $\Inn(\A)$ permute the components \eqref{eq:M2dRcomps} of square roots of $(-\unitm,-\unitm)$ in $\M(2,\R^2)$ as follows
\begin{align}
  \det(g)>0, \quad \det(g')>0: &\quad (c_1), (c_2,c_3), (c_4), \nonumber \\
  \det(g)>0, \quad\det(g')<0:  &\quad c_1 \rightarrow c_2 \rightarrow c_4 \rightarrow c_3 \rightarrow c_1, \nonumber \\
  \det(g)<0, \quad \det(g')>0: &\quad c_1 \rightarrow c_3 \rightarrow c_4 \rightarrow c_2 \rightarrow c_1, \nonumber \\
  \det(g)<0, \quad \det(g')<0: &\quad (c_1,c_4), (c_2), (c_3). 
  \label{eq:M2RoAut}
\end{align}
Consequently, the quotient of the group $\Aut(\A)$ by its neutral connected component is isomorphic to the group of isometries of a square in a Euclidean plane.

\section{Square roots of $-1$ in $\M(\lowercase{d},\HQ)$\label{sc:MdH}}
Let us first consider the easiest case $d = 1$, when $\A = \HQ$, e.g., of $\cl(0,2)$. The square roots of $-1$ in $\HQ$ are the quaternions $ai + bj + cij$ with $a^2 + b^2 + c^2 = 1$. They constitute a compact and connected manifold of dimension $2$. Every square root $f$ of $-1$ is conjugate with~$i$, i.e., there exists $v\in \HQ: v^{-1} f v = i \Leftrightarrow fv = vi$. 
If we set $v = -fi +1 = a +bij -cj +1$ we have 
$$
fv = -f^2i+f = f+i = (f(-i) +1)i = vi.
$$ 
$v$ is invertible, except when $f=-i$. But $i$ is conjugate with $-i$ because $ij = j(-i)$, hence, by transitivity $f$ is also conjugate with $-i$.

Here $\A =\M(d,\HQ)$, whence $\dim(A) = 4d^2$. The ring $\HQ$ is the algebra over $\R$ generated by two elements $i$ and $j$ such that $i^2 = j^2 = -1$ and $ji = -ij$. We identify $\C$ with the subalgebra generated by\footnote{This choice is usual and convenient.} $i$ alone.

The group $\G(\A)$ has only \textit{one} connected component. We shall soon prove that every square root of $-\unitm$ in $\A$ is conjugate with $i\unitm$. Therefore, the submanifold of square roots of 
$-\unitm$ is a conjugacy class, and it is connected. The centralizer of $i\unitm$ in $\A$ is the subalgebra of all matrices with entries in $\C$. The $\C$-dimension of $\Cent(i\unitm)$ is $d^2$, its 
$\R$-dimension is $2d^2$, and, consequently, the dimension \eqref{eq:dimconjclass} of the submanifold of square roots of $-\unitm$ is 
$4d^2 - 2d^2 = 2d^2 = \dim(\A)/2$.

Here $V = \HQ^d$ is treated as a (unitary) \textit{module} over $\HQ$ on the \textit{right} side: The product of a line vector ${}^tv=(x_1, x_2, \ldots, x_d) \in V$ by $y \in \HQ$ is 
${}^tv\,y=(x_1y, x_2y, \ldots, x_dy)$. Thus, every $f \in \A$ determines an $\HQ$-linear endomorphism of $V$: The matrix $f$ multiplies the column vector $v={}^t(x_1, x_2, \ldots, x_d)$ on the left side $v\mapsto fv$. Since $\C$ is a subring of $\HQ$, $V$ is also a vector space of dimension $2d$ over $\C$. The scalar $i$ always operates on the right side (like every scalar in $\HQ$). If $(e_1, e_2, \ldots, e_d)$ is an $\HQ$-basis of $V$, then $(e_1, e_1j, e_2, e_2j, \ldots, e_d, e_dj)$ is a $\C$-basis of $V$. Let $f$ be a square root of $-\unitm$, then the eigenvalues of $f$ in $\C$ are $+i$ or $-i$. If we treat $V$ as a $2d$ vector space over $\C$, it is the direct ($\C$-linear) sum of the eigenspaces
\be 
V^+ = \{v \in V \mid f(v) = vi \} \quad \mbox{and} \quad V^- = \{v \in V \mid f(v) = -vi \},
\ee 
representing $f$ as a $2d\times 2d$ $\C$-matrix w.r.t. the $\C$-basis of $V$, with $\C$-scalar eigenvalues (multiplied from the right): $\lambda_{\pm} = \pm i$.

Since $ij = -ji$, the multiplication $v \mapsto vj$ permutes $V^+$ and $V^-$, as $f(v) = \pm vi$ is mapped to $f(v)j = \pm vij = \mp (vj)i$. Therefore, if $(e_1, e_2, \ldots, e_r)$ is a $\C$-basis of 
$V^+$, then $(e_1j, e_2j, \ldots, e_rj)$ is a $\C$-basis of $V^-$, consequently $(e_1, e_1j, e_2, e_2j, \ldots, e_r, e_rj)$ is a $\C$-basis of $V$, and $(e_1, e_2, \ldots, e_{r=d})$ is an $\HQ$-basis of $V$. Since $f$ by $f(e_k) = e_ki$ for $k = 1, 2, \ldots, d$ operates on the $\HQ$-basis $(e_1, e_2, \ldots, e_{d})$ in the same way as $i\unitm$ on the natural $\HQ$-basis of $V$, we conclude
that $f$ and  $i\unitm$ are conjugate. 

Besides, $\Scal(i\unitm) = 0$ because $2i\unitm = [ j\unitm , ij\unitm ]\in [\A,\A]$, thus $i\unitm \notin \Zof(\A)$. Whence,\footnote{Compare the definition of $\Scal(f)$ in Section \ref{sc:intro}, remembering that in the current section the associated ring is $\HQ$.} 
\be 
\Scal(f) = 0 \quad \mbox{for every square root of} \quad -\unitm .
\label{eq:MdH2Scal0}
\ee 

These results are easily verified in the above example of  $d = 1$ when $\A = \HQ$.

\section{Square roots of $-1$ in $\M(\lowercase{d}, \HQ^2)$\label{sc:MdH2}}
Here, $\A = \M(d, \HQ^2) = \M(d,\HQ) \times \M(d,\HQ)$, whence $\dim(A) = 8d^2$. The group $\G(\A)$ has only one connected component (see Footnote \ref{fn:NoConComp}).

The square roots of $(-\unitm,-\unitm)$ in $\A$ are pairs of two square roots of $-\unitm$ in $\M(d,\HQ)$. Consequently, they constitute a unique conjugacy class which is connected and its dimension is $2\times 2d^2 = 4d^2 = \dim(\A)/2$.

For every $(f, f') \in \A$ we can write $\Scal(f) + \Scal(f') = 2 \,\Scal(f, f')$ and, similarly to \eqref{eq:M2dR2},
\be 
\Scal(f) - \Scal(f') = 2 \,\Spec(f, f') \quad \mbox{if} \quad \omega = (\unitm,-\unitm);
\ee
whence $\Scal(f, f') = \Spec(f, f') = 0$ if $(f, f')$ is a square root of $(-\unitm,-\unitm)$, compare with~\eqref{eq:MdH2Scal0}.

The group $\Aut(\A)$ has two\footnote{Compare Footnote \ref{fn:NoConComp}.} connected components; the neutral component is $\Inn(\A)$, and the other component contains the swap automorphism $(f, f') \mapsto (f', f)$.

The simplest example is $d=1$, $\A = \HQ^2$, where we have the identity pair $(1,1)$ representing the scalar part, the idempotents $(1,0)$, $(0,1)$, and $\omega$ as the pair $(1,-1)$.

$\A = \HQ^2$ is isomorphic to $\cl(0,3)$. The pseudoscalar $\omega = e_{123}$ has the square $\omega^2=+1$. The center of the algebra is $\{1,\omega\}$, and includes the idempotents 
$\epsilon_{\pm}=\frac12(1{\pm}\omega)$, $\epsilon_{\pm}^2 = \epsilon_{\pm}$, $\epsilon_{+}\epsilon_{-}=\epsilon_{-}\epsilon_{+}=0$. The basis of the algebra can thus be written as 
$\{\epsilon_{+}, e_1\epsilon_{+}, e_2\epsilon_{+}, e_{12}\epsilon_{+}, \epsilon_{-}, e_1\epsilon_{-}, e_2\epsilon_{-}, e_{12}\epsilon_{-}\}$ where the first (and the last) four elements form a basis of the subalgebra $\cl(0,2)$ isomorphic to $\HQ$.

\section{Square roots of $-1$ in $\M(2\lowercase{d}, \C)$}
\label{sc:M2dC}
The lowest dimensional example for $d=1$ is the Pauli matrix algebra $\A=\M(2,\C)$ isomorphic to the geometric algebra $\cl(3,0)$ of the 3D Euclidean space and $\cl(1,2)$. The $\cl(3,0)$ vectors 
$e_1,e_2,e_3$ correspond one-to-one to the Pauli matrices
\be
\sigma_1=\begin{pmatrix}0 & 1\\1 & 0\end{pmatrix}, 
\qquad
\sigma_2=\begin{pmatrix}0 & -i\\i & 0\end{pmatrix}, 
\qquad
\sigma_3=\begin{pmatrix}1 & 0 \\0 & -1\end{pmatrix},
\ee
with $\sigma_1\sigma_2 = i\sigma_3 =\begin{pmatrix}i & 0\\0 & -i\end{pmatrix}$. The element $\omega = \sigma_1\sigma_2\sigma_3= i\unitm$ represents the central pseudoscalar $e_{123}$ of $\cl(3,0)$ 
with square $\omega^2=-\unitm$. The Pauli algebra has the following idempotents
\be 
\epsilon_{1} = \sigma_1^2 = \unitm , \qquad \epsilon_0 = \frac{1}{2}(\unitm+\sigma_3), \qquad \epsilon_{-1} = \zerom \,.
\ee 
The idempotents correspond via 
\be 
  f= i(2\epsilon-\unitm),
  \label{eq:sr-1toidmpt}
\ee
to the square roots of $-\unitm$:
\be 
f_{1}=i\unitm = \begin{pmatrix}i & 0 \\0 & i\end{pmatrix},
\;
f_0 = i\sigma_3 = \begin{pmatrix}i & 0\\0 & -i\end{pmatrix},
\;
f_{-1}= -i\unitm = \begin{pmatrix}-i & 0\\0 & -i\end{pmatrix}, 
\ee 
where by \textit{complex} conjugation $f_{-1}=\overline{f_1}$. Let the idempotent $\epsilon_0^{\prime} = \frac{1}{2}(\unitm-\sigma_3)$ correspond to the matrix $f^{\prime}_0 = -i\sigma_3.$
We observe that $f_0$ is conjugate to $f^{\prime}_0 = \sigma_1^{-1} f_0 \sigma_1 = \sigma_1\sigma_2 = f_0$ using $\sigma_1^{-1} = \sigma_1$ but $f_1$ is not conjugate to $f_{-1}$. Therefore, only 
$f_{1}, f_{0}, f_{-1}$ lead to three distinct conjugacy classes of square roots of $-\unitm$ in $\M(2,\C)$. Compare Appendix \ref{AppendB} for the corresponding computations with CLIFFORD for Maple. 

In general, if $\A =\M(2d, \C)$, then $\dim(\A) = 8d^2$. The group $\G(\A)$ has one connected component. The square roots of $-\unitm$ in $\A$ are in bijection with the idempotents~$\epsilon$~\cite{AFPR:idem} according to \eqref{eq:sr-1toidmpt}. According\footnote{On the other hand it is clear that complex conjugation always leads to $f_- = \overline{f_+}$, where the overbar means complex conjugation in $\M(2d,\C)$ and Clifford conjugation in the isomorphic Clifford algebra $\cl(p,q)$. So either the trivial idempotent $\epsilon_- = 0$ is included in the bijection \eqref{eq:sr-1toidmpt} of idempotents and square roots of $-\unitm$, or alternatively the square root of $-\unitm$ with $\Spec(f_-)=-1$ is obtained from $f_- = \overline{f_+}$.} to \eqref{eq:sr-1toidmpt} and its inverse $\epsilon = \frac{1}{2}(\unitm -if)$ the square root of $-\unitm$ with $\Spec(f_-)=k/d=-1$, i.e. $k=-d$ (see below), always corresponds to the trival idempotent $\epsilon_- = 0$, and the square root of $-\unitm$ with $\Spec(f_+)=k/d=+1$, $k=+d$, corresponds to the identity idempotent $\epsilon_+ = \unitm$. 

If $f$ is a square root of $-\unitm$, then $V = \C^{2d}$ is the direct sum of the eigenspaces\footnote{The following theorem is sufficient for a matrix $f$ in $\M(m,\K)$, if $\K$ is a (commutative) field. The matrix $f$ is diagonalizable if and only if $P(f)=0$ for some polynomial $P$ that has only simple roots, all of them in the field $\K$. (This implies that $P$ is a multiple of the minimal polynomial, but we do not need to know whether $P$ is or is not the minimal polynomial).} associated with the eigenvalues $i$ and $-i$. There is an integer $k$ such that the dimensions of the eigenspaces are respectively $d + k$ and $d - k$. Moreover, $-d \leq k \leq d$. Two square roots of $-\unitm$ are conjugate if and only if they give the same integer $k$. Then, all elements of 
$\Cent(f)$ consist of diagonal block matrices with $2$ square blocks of $(d+k)\times(d+k)$ matrices and $(d-k)\times(d-k)$ matrices. Therefore, the $\C$-dimension of $\Cent(f)$ is $(d+k)^2 + (d-k)^2$. Hence the $\R$-dimension \eqref{eq:dimconjclass} of the conjugacy class of $f$:
\be
8d^2 - 2(d+k)^2 - 2(d-k)^2 = 4 (d^2 - k^2).
\label{eq:M2dCRdim}
\ee 
Also, from the equality $\trace{f} = (d+k)i-(d-k)i=2ki$ we deduce that $\Scal(f) = 0$ and that $\Spec(f) = (2ki)/(2di) = k/d$ if $\omega = i\unitm$ (whence $\trace{\omega} = 2di$).

As announced on page \pageref{pg:ordinary}, we consider that a square root of $-\unitm$ is \textit{ordinary} if the associated integer $k$ vanishes, and that it is \textit{exceptional} if $k \neq 0$ . Thus the following assertion is true in all cases: the ordinary square roots of $-\unitm$ in $\A$ constitute one conjugacy class of dimension $\dim(\A)/2$ which has as many connected components as $\G(\A)$, and the equality $\Spec(f) = 0$ holds for every ordinary square root of $-\unitm$ when the linear form Spec exists. All conjugacy classes of exceptional square roots of $-\unitm$ have a dimension $< \dim(\A)/2$.

All square roots of $-\unitm$ in $\M(2d,\C)$ constitute $(2d+1)$ conjugacy classes\footnote{Two conjugate (similar) matrices have the same eigenvalues and the same trace. This suffices to recognize that $2d+1$ conjugacy classes are obtained.} which are also the connected components of the submanifold of square roots of $-\unitm$ because of the equality $\Spec(f) = k/d$, which is conjugacy class specific.

When $\A =\M(2d,\C)$, the group $\Aut(\A)$ is larger than $\Inn(\A)$ since it contains the complex conjugation (that maps every entry of a matrix to the conjugate complex number). It is clear that the class of ordinary square roots of $-\unitm$ is invariant by complex conjugation. But the class associated with an integer $k$ other than 0 is mapped by complex conjugation to the class associated with $-k$. In particular the complex conjugation maps the class $\{\omega\}$ (associated with $k = d$) to the class $\{-\omega\}$ associated with $k = -d$.

All these observations can easily verified for the above example of $d=1$ of the Pauli matrix algebra $\A=\M(2,\C)$. For $d=2$ we have the isomorphism of $\A =\M(4,\C)$ with $\cl(0,5)$,  
$\cl(2,3)$ and  $\cl(4,1)$. While $\cl(0,5)$ is important in Clifford analysis, $\cl(4,1)$ is both the geometric algebra of the Lorentz space $\R^{4,1}$ and the conformal geometric algebra of 3D Euclidean geometry. Its set of square roots of $-\unitm$ is therefore of particular practical interest. 

\begin{ex}
Let $\cl(4,1) \cong  \A$ where $\A = \M(4,\C)$ for $d=2$. The $\cl(4,1)$ 1-vectors can be represented\footnote{For the computations of this example in the Maple package CLIFFORD we
have used the identification $i=e_{23}$. Yet the results obtained for the
square roots of $-\unitm$ are independent of this setting (we can alternatively
use, e.g., $i=e_{12345}$, or the imaginary unit $i \in \C$), as
can easily be checked for $f_1$ of 
(\ref{eq:f31}),
$f_0$ of 
(\ref{eq:f32})
and $f_{-1}$ of 
(\ref{eq:Cl41f-1})
by only assuming the standard Clifford product rules for $e_1$ to $e_5$. } by the following matrices:
\begin{gather}
e_1=\begin{pmatrix} 1 & 0 & 0 & 0\\ 0 & -1 & 0 & 0\\0 & 0 & -1 & 0\\0 & 0 & 0 & 1 \end{pmatrix},\; 
e_2=\begin{pmatrix} 0 & 1 & 0 & 0\\ 1 &  0 & 0 & 0\\0 & 0 &  0 & 1\\0 & 0 & 1 & 0 \end{pmatrix},\; 
e_3=\begin{pmatrix} 0 & -i & 0 & 0\\ i & 0 & 0 & 0\\0 & 0 &  0 & -i\\0 & 0 & i & 0 \end{pmatrix}, \notag \\
e_4=\begin{pmatrix} 0 & 0 & 1 & 0\\ 0 & 0 & 0 & -1\\1 & 0 &  0 & 0\\0 & -1 & 0 & 0 \end{pmatrix},\; 
e_5=\begin{pmatrix} 0 & 0 & -1 & 0\\ 0 & 0 & 0 & 1\\1 & 0 & 0 & 0\\0 & -1 & 0 & 0 \end{pmatrix}.
\label{eq:matr} 
\end{gather}
We find five conjugacy classes of roots $f_k$ of $-\unitm$ in $\cl(4,1)$ for $k \in \{0,\pm 1, \pm 2\}$: four exceptional and one ordinary. Since $f_k$ is a root of $p(t) = t^2+1$ which factors over $\C$ into $(t-i)(t+i)$, the minimal polynomial $m_k(t)$ of $f_k$ is one of the following: $t-i,$ $t+i,$ or $(t-i)(t+i).$ Respectively, there are three classes of characteristic polynomial $\Delta_k(t)$ of the matrix $\mathcal{F}_k$ in $\M(4,\C)$ which corresponds to $f_k$, namely, $(t-i)^4,$  $(t+i)^4,$ and $(t-i)^{n_1}(t+i)^{n_2},$ where $n_1+n_2 = 2d=4$ and 
$n_1=d+k=2+k$, $n_2=d-k=2-k$. As predicted by the above discussion, the ordinary root corresponds to $k=0$ whereas the exceptional roots correspond to $k \neq 0.$
\begin{enumerate}
\item For $k=2,$ we have $\Delta_2(t)= (t-i)^4,$ $m_2(t)= t-i,$ and so $\mathcal{F}_2 = \mathrm{diag}(i,i,i,i)$ which in the above representation~(\ref{eq:matr}) corresponds to the non-trivial central element $f_2=\omega = e_{12345}.$ Clearly, $\mathrm{Spec}(f_2)= 1 = \frac{k}{d}$; $\mathrm{Scal}(f_2)= 0;$ the $\C$-dimension of the centralizer $\mathrm{Cent}(f_2)$ is $16$; and the $\mathbb{R}$-dimension of the conjugacy class of $f_2$ is zero as it contains only $f_2$ since $f_2 \in \mathrm{Z}(\A).$ Thus, the $\R$-dimension of the class is again zero in agreement with~(\ref{eq:M2dCRdim}). 
\item For $k=-2,$ we have $\Delta_{-2}(t)= (t+i)^4,$ $m_{-2}(t)= t+i,$ and $\mathcal{F}_{-2} = \mathrm{diag}(-i,-i,-i,-i)$ which corresponds to the central element $f_{-2}=-\omega = -e_{12345}.$ Again, $\mathrm{Spec}(f_{-2})= -1 = \frac{k}{d}$; $\mathrm{Scal}(f_{-2})= 0;$ the $\C$-dimension of the centralizer $\mathrm{Cent}(f_{-2})$ is $16$ and the conjugacy class of $f_{-2}$ contains only $f_{-2}$ since $f_{-2} \in \mathrm{Z}(\A).$ Thus, the $\R$-dimension of the class is again zero in agreement with~(\ref{eq:M2dCRdim}). 
\item For $k \neq \pm 2,$ we consider three subcases when $k=1,$ $k=0,$ and $k=-1.$ When $k=1,$ then $\Delta_{1}(t)=(t-i)^3(t+i)$ and $m_{1}(t)= (t-i)(t+i).$ Then the root 
$\mathcal{F}_1 = \mathrm{diag}(i,i,i,-i)$ corresponds to
\begin{gather}
f_{1} = \frac12(e_{23}+e_{123}-e_{2345}+e_{12345}).
\label{eq:f31} 
\end{gather}
Note that $\mathrm{Spec}(f_1) = \frac12 = \frac{k}{d}$ so $f_1$ is an exceptional root of~$-\unitm$.

\noindent
When $k=0,$ then $\Delta_{0}(t)= (t-i)^2(t+i)^2$ and $m_{0}(t)=(t-i)(t+i)$. Thus the root of $-\unitm$ in this case is $\mathcal{F}_{0} = \mathrm{diag}(i,i,-i,-i)$ which corresponds to just
\begin{gather}
  f_{0} = e_{123}.
  \label{eq:f32} 
\end{gather}
Note that $\mathrm{Spec}(f_0) = 0$ thus $f_0=e_{123}$ is an ordinary root of $-\unitm$.

\noindent
When $k=-1,$ then $\Delta_{-1}(t)=(t-i)(t+i)^3$ and $m_{-1}(t)=(t-i)(t+i)$. Then, the root of $-\unitm$ in this case is $\mathcal{F}_{-1}=\mathrm{diag}(i,-i,-i,-i)$ which corresponds to
\begin{gather}
  f_{-1}=\frac12(e_{23}+e_{123}+e_{2345}-e_{12345}).
  \label{eq:Cl41f-1}
\end{gather}
Since $\mathrm{Scal}(f_{-1}) = - \frac12 = \frac{k}{d},$ we gather that $f_{-1}$ is an exceptional root. 

As expected, we can also see that the roots $\omega$ and $-\omega$ are related via the grade involution whereas $f_{1}=-\tilde{f}_{-1}$ where $\tilde{\phantom{u}}$ denotes the reversion in 
$\cl(4,1).$
\end{enumerate}
\label{example4}
\end{ex}

\begin{ex}
Let $\cl(0,5) \cong  \A$ where $\A = \M(4,\C)$ for $d=2$. The $\cl(0,5)$ 1-vectors can be represented\footnote{For the computations of this example in the Maple package CLIFFORD we
have used the identification $i=e_{3}$. Yet the results obtained for the
square roots of $-\unitm$ are independent of this setting (we can alternatively
use, e.g., $i=e_{12345}$, or the imaginary unit $i \in \C$), as
can easily be checked for $f_1$ of 
(\ref{eq:f3105}),
$f_0$ of 
(\ref{eq:f3205})
 and $f_{-1}$ of 
(\ref{eq:Cl05f-1})
by only assuming the standard Clifford product rules for $e_1$ to $e_5$. } by the following matrices:
\begin{gather}
e_1=\begin{pmatrix} 0 & -1 & 0 & 0\\ 1 & 0 & 0 & 0\\0 & 0 & 0 & -1\\0 & 0 & 1 & 0 \end{pmatrix},\, 
e_2=\begin{pmatrix} 0 & -i & 0 & 0\\ -i &  0 & 0 & 0\\0 & 0 &  0 & -i\\0 & 0 & -i & 0 \end{pmatrix},\, 
e_3=\begin{pmatrix} -i & 0 & 0 & 0\\ 0 & i & 0 & 0\\0 & 0 &  i & 0\\0 & 0 & 0 & -i \end{pmatrix}, \notag \\
e_4=\begin{pmatrix} 0 & 0 & -1 & 0\\ 0 & 0 & 0 & 1\\1 & 0 &  0 & 0\\0 & -1 & 0 & 0 \end{pmatrix},\, 
e_5=\begin{pmatrix} 0 & 0 & -i & 0\\ 0 & 0 & 0 & i\\-i & 0 & 0 & 0\\0 & i & 0 & 0 \end{pmatrix},
\label{eq:matr05} 
\end{gather}
Like for $\cl(4,1),$ we have five conjugacy classes of the roots $f_k$ of $-\unitm$ in $\cl(0,5)$ for $k \in \{0,\pm 1, \pm 2\}$: four exceptional and one ordinary. Using the same notation as in Example~\ref{example4}, we find the following representatives of the conjugacy classes.
\begin{enumerate}
\item For $k=2,$ we have $\Delta_2(t)= (t-i)^4,$ $m_2(t)= t-i,$ and $\mathcal{F}_2 = \mathrm{diag}(i,i,i,i)$ which in the above representation~(\ref{eq:matr05}) corresponds to the non-trivial central element $f_2=\omega = e_{12345}.$ Then, $\mathrm{Spec}(f_2)= 1 = \frac{k}{d}$; $\mathrm{Scal}(f_2)= 0;$ the $\C$-dimension of the centralizer $\mathrm{Cent}(f_2)$ is $16$; and the $\mathbb{R}$-dimension of the conjugacy class of $f_2$ is zero as it contains only $f_2$ since $f_2 \in \mathrm{Z}(\A).$ Thus, the $\R$-dimension of the class is again zero in agreement with~(\ref{eq:M2dCRdim}). 
\item For $k=-2,$ we have $\Delta_{-2}(t)= (t+i)^4,$ $m_{-2}(t)= t+i,$ and $\mathcal{F}_{-2} = \mathrm{diag}(-i,-i,-i,-i)$ which corresponds to the central element $f_{-2=}-\omega = -e_{12345}.$ Again, $\mathrm{Spec}(f_{-2})= -1 = \frac{k}{d}$; $\mathrm{Scal}(f_{-2})= 0;$ the $\C$-dimension of the centralizer $\mathrm{Cent}(f_{-2})$ is $16$ and the conjugacy class of $f_{-2}$ contains only $f_{-2}$ since $f_{-2} \in \mathrm{Z}(\A).$ Thus, the $\R$-dimension of the class is again zero in agreement with~(\ref{eq:M2dCRdim}). 
\item For $k \neq \pm 2,$ we consider three subcases when $k=1,$ $k=0,$ and $k=-1.$ When $k=1,$ then $\Delta_{1}(t)=(t-i)^3(t+i)$ and $m_{1}(t)= (t-i)(t+i).$ Then the root 
$\mathcal{F}_1 = \mathrm{diag}(i,i,i,-i)$ corresponds to
\begin{gather}
f_{1} = \frac12(e_{3}+e_{12}+e_{45}+e_{12345}).
  \label{eq:f3105} 
\end{gather}
Since $\mathrm{Spec}(f_1) = \frac12 = \frac{k}{d}$, $f_1$ is an exceptional root of~$-\unitm$.

\noindent
When $k=0,$ then $\Delta_{0}(t)= (t-i)^2(t+i)^2$ and $m_{0}(t)=(t-i)(t+i)$. Thus the root of $-\unitm$ is this case is $\mathcal{F}_{0} = \mathrm{diag}(i,i,-i,-i)$ which corresponds to just
\begin{gather}
  f_{0} = e_{45}.
  \label{eq:f3205} 
\end{gather}
Note that $\mathrm{Spec}(f_0) = 0$ thus $f_0=e_{45}$ is an ordinary root of $-\unitm$.

\noindent
When $k=-1,$ then $\Delta_{-1}(t)=(t-i)(t+i)^3$ and $m_{-1}(t)=(t-i)(t+i)$. Then, the root of $-\unitm$ in this case is $\mathcal{F}_{-1}=\mathrm{diag}(i,-i,-i,-i)$ which corresponds to
\begin{gather}
  f_{-1}=\frac12(-e_{3}+e_{12}+e_{45}-e_{12345}).
  \label{eq:Cl05f-1}
\end{gather}
Since $\mathrm{Scal}(f_{-1}) = - \frac12 = \frac{k}{d},$ we gather that $f_{-1}$ is an exceptional root. 

Again we can see that the roots $f_2$ and $f_{-2}$ are related via the grade involution whereas $f_1=-\tilde{f}_{-1}$ where $\tilde{\phantom{u}}$ denotes the reversion in $\cl(0,5).$
\end{enumerate}
\end{ex}

\begin{ex}
Let $\cl(7,0) \cong  \A$ where $\A = \M(8,\C)$ for $d=4$. We have nine conjugacy classes of roots $f_k$ of $-\unitm$ for $k \in \{0,\pm 1, \pm 2\, \pm 3\, \pm 4\}.$ Since $f_k$ is a root of a polynomial $p(t) = t^2+1$ which factors over $\C$ into $(t-i)(t+i)$, its minimal polynomial $m(t)$ will be one of the following: $t-i,$ $t+i,$ or $(t-i)(t+i)=t^2+1.$ 

Respectively, each conjugacy class is characterized by a characteristic polynomial $\Delta_k(t)$ of the matrix $M_k \in \M(8,\C)$ which represents $f_k$. Namely, we have 
$$
\Delta_k(t) = (t-i)^{n_1}(t+i)^{n_2},
$$  
where $n_1+n_2 = 2d=8$ and $n_1=d+k=4+k$ and $n_2=d-k=4-k$. The ordinary root of $-\unitm$ corresponds to $k=0$ whereas the exceptional roots correspond to $k \neq 0.$

\begin{enumerate}
\item When $k=4,$ we have $\Delta_4(t)= (t-i)^8,$ $m_4(t)= t-i,$ and $\mathcal{F}_4 = \mathrm{diag}(\overbrace{i,\ldots, i}^{8})$ which in the representation used by CLIFFORD~\cite{AF:CLIFFORD} corresponds to the non-trivial central element $f_4=\omega = e_{1234567}.$ Clearly, $\mathrm{Spec}(f_4)= 1 = \frac{k}{d}$; $\mathrm{Scal}(f_4)= 0;$ the $\C$-dimension of the centralizer 
$\mathrm{Cent}(f_4)$ is $64$; and the $\mathbb{R}$-dimension of the conjugacy class of $f_4$ is zero since $f_4 \in \mathrm{Z}(\A).$ Thus, the $\R$-dimension of the class is again zero in agreement 
with~(\ref{eq:M2dCRdim}). 
\item When $k=-4,$ we have $\Delta_{-4}(t)= (t+i)^8,$ $m_{-4}(t)= t+i,$ and $\mathcal{F}_{-4} = \mathrm{diag}(\overbrace{-i,\ldots, -i}^{8})$ which corresponds to $f_{-4}=-\omega = -e_{1234567}.$ 
Again, $\mathrm{Spec}(f_{-4})= -1 = \frac{k}{d}$; $\mathrm{Scal}(f_{-4})= 0;$ the $\C$-dimension of the centralizer $\mathrm{Cent}(f)$ is $64$ and the conjugacy class of $f_{-4}$ contains only $f_{-4}$ since $f_{-4} \in \mathrm{Z}(\A).$ Thus, the $\R$-dimension of the class is again zero in agreement with~(\ref{eq:M2dCRdim}). 
\item When $k \neq \pm 4,$ we consider seven subcases when $k=\pm 3,$ $k=\pm 2,$ $k=\pm 1,$ and $k=0.$ 

\noindent
When $k=3,$ then $\Delta_{3}(t)=(t-i)^7(t+i)$ and $m_{3}(t)= (t-i)(t+i).$ Then the root $\mathcal{F}_{3} = \mathrm{diag}(\overbrace{i,\ldots,i}^{7},-i)$ corresponds to
\begin{gather}
f_{3}= \frac14(e_{23}-e_{45}+e_{67}-e_{123}+e_{145}-e_{167}+e_{234567}+3e_{1234567}).
\label{eq:f370} 
\end{gather}
Since $\mathrm{Spec}(f_3) = \frac34 = \frac{k}{d}$, $f_3$ is an exceptional root of~$-\unitm$.

\noindent
When $k=2,$ then $\Delta_{2}(t)=(t-i)^6(t+i)^2$ and $m_{2}(t)= (t-i)(t+i).$ Then the root $\mathcal{F}_{2} = \mathrm{diag}(\overbrace{i,\ldots,i}^{6},-i,-i)$ corresponds to
\begin{gather}
f_{2}= \frac12(e_{67}-e_{45}-e_{123}+e_{1234567}).
\label{eq:f270} 
\end{gather}
Since $\mathrm{Spec}(f_2) = \frac12 = \frac{k}{d}$, $f_2$ is also an exceptional root.

\noindent
When $k=1,$ then $\Delta_{1}(t)=(t-i)^5(t+i)^3$ and $m_{1}(t)= (t-i)(t+i).$ Then the root $\mathcal{F}_{1} = \mathrm{diag}(\overbrace{i,\ldots,i}^{5},-i,-i,-i)$ corresponds to
\begin{gather}
f_{1}= \frac14(e_{23}-e_{45}+3e_{67}-e_{123}+e_{145}+e_{167}-e_{234567}+e_{1234567}).
\label{eq:f170} 
\end{gather}
Since $\mathrm{Spec}(f_1) = \frac14 = \frac{k}{d}$, $f_1$ is another exceptional root.

\noindent
When $k=0,$ then $\Delta_{0}(t)=(t-i)^4(t+i)^4$ and $m_{0}(t)= (t-i)(t+i).$ Then the root $\mathcal{F}_{0} = \mathrm{diag}(i,i,i,i,-i,-i,-i,-i)$ corresponds to
\begin{gather}
f_{0}= \frac12(e_{23}-e_{45}+e_{67}-e_{234567}).
\label{eq:f070} 
\end{gather}
Since $\mathrm{Spec}(f_0) = 0 = \frac{k}{d}$, we see that $f_0$ is an ordinary root of $-\unitm$.

\noindent
When $k=-1,$ then $\Delta_{-1}(t)=(t-i)^3(t+i)^5$ and $m_{-1}(t)= (t-i)(t+i).$ Then the root $\mathcal{F}_{-1} = \mathrm{diag}(i,i,i,\overbrace{-i,\ldots,-i}^{5})$ corresponds to
\begin{gather}
f_{-1}= \frac14(e_{23}-e_{45}+3e_{67}+e_{123}-e_{145}-e_{167}-e_{234567}-e_{1234567}).
\label{eq:fm170} 
\end{gather}
Thus, $\mathrm{Spec}(f_{-1}) = -\frac14 = \frac{k}{d}$ and so $f_{-1}$ is another exceptional root.

\noindent
When $k=-2,$ then $\Delta_{-2}(t)=(t-i)^2(t+i)^6$ and $m_{-2}(t)= (t-i)(t+i).$ Then the root $\mathcal{F}_{-2} = \mathrm{diag}(i,i,\overbrace{-i,\ldots,-i}^{6})$ corresponds to
\begin{gather}
f_{-2}= \frac12(e_{67}-e_{45}+e_{123}-e_{1234567}).
\label{eq:fm270} 
\end{gather}
Since $\mathrm{Spec}(f_{-2}) = -\frac12 = \frac{k}{d}$, we see that $f_{-2}$ is also an exceptional root.

\noindent
When $k=-3,$ then $\Delta_{-3}(t)=(t-i)(t+i)^7$ and $m_{-3}(t)= (t-i)(t+i).$ Then the root $\mathcal{F}_{-3} = \mathrm{diag}(i,\overbrace{-i,\ldots,-i}^{7})$ corresponds to
\begin{gather}
f_{-3}= \frac14(e_{23}-e_{45}+e_{67}+e_{123}-e_{145}+e_{167}+e_{234567}-3e_{1234567}).
\label{eq:fm370} 
\end{gather}
Again, $\mathrm{Spec}(f_{-3}) = -\frac34 = \frac{k}{d}$ and so $f_{-3}$ is another exceptional root of~$-\unitm$.

As expected, we can also see that the roots $\omega$ and $-\omega$ are related via the reversion whereas $f_{3}=-\bar{f}_{-3}$, $f_{2}=-\bar{f}_{-2}$, $f_{1}=-\bar{f}_{-1}$ where $\bar{\phantom{u}}$ denotes the conjugation in $\cl(7,0).$
\end{enumerate}
\end{ex}

\section{Conclusions}
We proved that in all cases $\Scal(f) = 0$ for every square root of $-\unitm$ in $\A$ isomorphic to $\cl(p,q)$. We distinguished \textit{ordinary} square roots of $-\unitm$, and \textit{exceptional} ones. 

In all cases the ordinary square roots $f$ of $-\unitm$ constitute a unique conjugacy class of dimension $\dim(\A)/2$ which has as many connected components as the group $\G(\A)$ of invertible elements in $\A$. Furthermore, we have $\Spec(f) = 0$ (zero pseudoscalar part) if the associated ring is $\R^2$, $\HQ^2$, or $\C$. The exceptional square roots of $-\unitm$ \textit{only} exist if 
$\A \cong \M(2d,\C)$ (see Section~\ref{sc:M2dC}). 

For $\A=\M(2d,\R)$ of Section~\ref{sc:M2dR}, the centralizer and the conjugacy class of a square root $f$ of $-\unitm$ both have $\R$-dimension $2d^2$ with two connected components, pictured in 
Fig.~\ref{fg:M2dR} for $d=1$.

For $\A=\M(2d,\R^2)=\M(2d,\R)\times\M(2d,\R)$ of Section~\ref{sc:M2dR2}, the square roots of $(-\unitm,-\unitm)$ are pairs of two square roots of $-\unitm$ in $\M(2d,\R)$. They constitute a unique conjugacy class with four connected components, each of dimension $4d^2$. Regarding the four connected components, the group $\Inn(\A)$ induces the permutations of the Klein group whereas the quotient group $\Aut(\A)/\Inn(\A)$ is isomorphic to the group of isometries of a Euclidean square in 2D. 

For $\A=\M(d,\HQ)$ of Section~\ref{sc:MdH}, the submanifold of the square roots $f$ of $-\unitm$ is a single connected conjugacy class of $\R$-dimension $2d^2$ equal to the $\R$-dimension of the centralizer of every $f$. The easiest example is $\HQ$ itself for $d=1$.

For $\A=\M(d,\HQ^2)=\M(2d,\HQ)\times\M(2d,\HQ)$ of Section~\ref{sc:MdH2}, the square roots of $(-\unitm,-\unitm)$ are pairs of two square roots $(f,f')$ of $-\unitm$ in $\M(2d,\HQ)$ and constitute a unique connected conjugacy class of $\R$-dimension $4d^2$. The group $\Aut(\A)$ has two connected components: the neutral component $\Inn(\A)$ connected to the identity and the second component containing the swap automorphism $(f,f')\mapsto (f',f)$. The simplest case for $d=1$ is $\HQ^2$ isomorphic to $\cl(0,3)$.

For $\A=\M(2d,\C)$ of Section~\ref{sc:M2dC}, the square roots of $-\unitm$ are in bijection to the idempotents. First, the ordinary square roots of $-\unitm$ (with $k=0$) constitute a conjugacy class of $\R$-dimension $4d^2$ of a single connected component which is invariant under $\Aut(\A)$. Second, there are $2d$ conjugacy classes of exceptional square roots of $-\unitm$, each composed of a single connected component, characterized by equality $\Spec(f) = k/d$ (the pseudoscalar coefficient) with $\pm k \in \{1, 2, \ldots, d\}$, and their $\R$-dimensions are $4(d^2-k^2)$. The group 
$\Aut(\A)$ includes conjugation of the pseudoscalar $\omega \mapsto -\omega$ which maps the conjugacy class associated with $k$ to the class associated with $-k$. The simplest case for $d=1$ is the Pauli matrix algebra isomorphic to the geometric algebra $\cl(3,0)$ of 3D Euclidean space $\R^3$, and to complex biquaternions \cite{SJS:Biqroots}. 

Section~\ref{sc:M2dC} includes explicit examples for $d=2$: $\cl(4,1)$ and $\cl(0,5)$, and for $d=4$: $\cl(7,0)$. Appendix~\ref{AppendA} summarizes the square roots of $-\unitm$ in all 
$\cl(p,q) \cong \M(2d,\C)$ for $d=1,2,4$. Appendix~\ref{AppendB} contains details on how square roots of $-\unitm$ can be computed using the package CLIFFORD for Maple. 

Among the many possible \textit{applications} of this research, the possibility of \textit{new integral transformations} in Clifford analysis is very promising. This field thus obtains essential algebraic information, which can e.g., be used to create \textit{steerable} transformations, which may be steerable within a connected component of a submanifold of square roots of $-\unitm$. 

\appendix
\section{Summary of roots of $-\unitm$ in $\cl(\lowercase{p,q}) \cong \M(2\lowercase{d},\C)$ for $\lowercase{d}=1,2,4$}
\label{AppendA}

In this appendix we summarize roots of $-\unitm$ for Clifford algebras $\cl(p,q) \cong \M(2d,\C)$ for $d=1,2,4$. These roots have been computed with CLIFFORD~\cite{AF:CLIFFORD}. Maple~\cite{Maple} worksheets written to derive these roots are posted at~\cite{worksheets}.

\begin{table}[htb]
\label{tab:t1}
\begin{center}
\renewcommand{\arraystretch}{1.3}
\begin{tabular}{|c|l|l|}\hline
$k$ & $f_k$ & $\Delta_k(t)$  \\\hline
$1$ & $\omega =e_{123}$ & $(t-i)^2$   \\\hline
$0$ & $e_{23}$ & $(t-i)(t+i)$   \\\hline
$-1$ & $-\omega = -e_{123}$ & $(t+i)^2$   \\\hline
\end{tabular}
\end{center}
\caption{Square roots of $-\unitm$ in $\cl(3,0)\cong \M(2,\C)$, $d=1$}
\label{table1}
\end{table}

\begin{table}[htb]
\label{tab:t2}
\begin{center}
\renewcommand{\arraystretch}{1.3}
\begin{tabular}{|c|l|l|}\hline
$k$ & $f_k$ & $\Delta_k(t)$  \\\hline
$2$ & $\omega=e_{12345}$ & $(t-i)^4$  \\\hline
$1$ & $\frac12(e_{23}+e_{123}-e_{2345}+e_{12345})$ & $(t-i)^3(t+i)$   \\\hline
$0$ & $e_{123}$ & $(t-i)^2(t+i)^2$   \\\hline
$-1$ & $\frac12(e_{23}+e_{123}+e_{2345}-e_{12345})$ & $(t-i)(t+i)^3$   \\\hline
$-2$ & $-\omega=-e_{12345}$ & $(t+i)^4$   \\\hline
\end{tabular}
\end{center}
\caption{Square roots of $-\unitm$ in $\cl(4,1)\cong \M(4,\C)$, $d=2$}
\label{table2}
\end{table}

\begin{table}[htb]
\label{tab:t3}
\begin{center}
\renewcommand{\arraystretch}{1.3}
\begin{tabular}{|c|l|l|}\hline
$k$ & $f_k$ & $\Delta_k(t)$  \\\hline
$2$ & $\omega=e_{12345}$ & $(t-i)^4$  \\\hline
$1$ & $\frac12(e_{3}+e_{12}+e_{45}+e_{12345})$ & $(t-i)^3(t+i)$   \\\hline
$0$ & $e_{45}$ & $(t-i)^2(t+i)^2$   \\\hline
$-1$ & $\frac12(-e_{3}+e_{12}+e_{45}-e_{12345})$ & $(t-i)(t+i)^3$   \\\hline
$-2$ & $-\omega=-e_{12345}$ & $(t+i)^4$   \\\hline
\end{tabular}
\end{center}
\caption{Square roots of $-\unitm$ in $\cl(0,5)\cong \M(4,\C)$, $d=2$}
\label{table3}
\end{table}

\begin{table}[htb]
\label{tab:t4}
\begin{center}
\renewcommand{\arraystretch}{1.3}
\begin{tabular}{|c|l|l|}\hline
$k$ & $f_k$ & $\Delta_k(t)$  \\\hline
$2$ & $\omega=e_{12345}$ & $(t-i)^4$  \\\hline
$1$ & $\frac12(e_{3}+e_{134}+e_{235}+\omega)$ & $(t-i)^3(t+i)$   \\\hline
$0$ & $e_{134}$ & $(t-i)^2(t+i)^2$   \\\hline
$-1$ & $\frac12(-e_{3}+e_{134}+e_{235}-\omega)$ & $(t-i)(t+i)^3$   \\\hline
$-2$ & $-\omega=-e_{12345}$ & $(t+i)^4$   \\\hline
\end{tabular}
\end{center}
\caption{Square roots of $-\unitm$ in $\cl(2,3)\cong \M(4,\C)$, $d=2$}
\label{table4}
\end{table}

\begin{table}[htb]
\label{tab:t5}
\begin{center}
\renewcommand{\arraystretch}{1.3}
\begin{tabular}{|c|p{2.50in}|l|}\hline
$k$ & $f_k$ & $\Delta_k(t)$  \\\hline
$4$ & $\omega=e_{1234567}$ & $(t-i)^8$  \\\hline
$3$ & $\frac14(e_{23}-e_{45}+e_{67}-e_{123}+e_{145}\newline \hspace*{20ex}-e_{167}+e_{234567}+3\omega)$ & $(t-i)^7(t+i)$  \\\hline
$2$ & $\frac12(e_{67}-e_{45}-e_{123}+\omega)$ & $(t-i)^6(t+i)^2$  \\\hline
$1$ & $\frac14(e_{23}-e_{45}+3e_{67}-e_{123}+e_{145}\newline \hspace*{20ex}+e_{167}-e_{234567}+\omega)$ & $(t-i)^5(t+i)^3$  \\\hline
$0$ & $\frac12(e_{23}-e_{45}+e_{67}-e_{234567})$ & $(t-i)^4(t+i)^4$   \\\hline
$-1$ & $\frac14(e_{23}-e_{45}+3e_{67}+e_{123}-e_{145}\newline \hspace*{20ex}-e_{167}-e_{234567}-\omega)$ & $(t-i)^3(t+i)^5$ \\\hline
$-2$ & $\frac12(e_{67}-e_{45}+e_{123}-\omega)$ & $(t-i)^2(t+i)^6$   \\\hline
$-3$ & $\frac14(e_{23}-e_{45}+e_{67}+e_{123}-e_{145}\newline \hspace*{20ex}+e_{167}+e_{234567}-3\omega)$ & $(t-i)(t+i)^7$   \\\hline
$-4$ & $-\omega=-e_{1234567}$ & $(t+i)^8$   \\\hline
\end{tabular}
\end{center}
\caption{Square roots of $-\unitm$ in $\cl(7,0)\cong \M(8,\C)$, $d=4$}
\label{table5}
\end{table}

\begin{table}[htb]
\label{tab:t6}
\begin{center}
\renewcommand{\arraystretch}{1.3}
\begin{tabular}{|c|p{2.50in}|l|}\hline
$k$ & $f_k$ & $\Delta_k(t)$  \\\hline
$4$ & $\omega=e_{1234567}$ & $(t-i)^8$  \\\hline
$3$ & $\frac14(e_{4}-e_{23}-e_{56}+e_{1237}+e_{147}\newline \hspace*{20ex}+e_{1567}-e_{23456}+3\omega)$ & $(t-i)^7(t+i)$  \\\hline
$2$ & $\frac12(-e_{23}-e_{56}+e_{147}+\omega)$ & $(t-i)^6(t+i)^2$  \\\hline
$1$ & $\frac14(-e_{4}-e_{23}-3e_{56}-e_{1237}+e_{147}\newline \hspace*{20ex}+e_{1567}-e_{23456}+\omega)$ & $(t-i)^5(t+i)^3$  \\\hline
$0$ & $\frac12(e_{4}+e_{23}+e_{56}+e_{23456})$ & $(t-i)^4(t+i)^4$   \\\hline
$-1$ & $\frac14(-e_{4}-e_{23}-3e_{56}+e_{1237}-e_{147}\newline \hspace*{20ex}-e_{1567}-e_{23456}-\omega)$ & $(t-i)^3(t+i)^5$ \\\hline
$-2$ & $\frac12(-e_{23}-e_{56}-e_{147}-\omega)$ & $(t-i)^2(t+i)^6$   \\\hline
$-3$ & $\frac14(e_{4}-e_{23}-e_{56}-e_{1237}-e_{147}\newline \hspace*{20ex}-e_{1567}-e_{23456}-3\omega)$ & $(t-i)(t+i)^7$   \\\hline
$-4$ & $-\omega=-e_{1234567}$ & $(t+i)^8$   \\\hline
\end{tabular}
\end{center}
\caption{Square roots of $-\unitm$ in $\cl(1,6)\cong \M(8,\C)$, $d=4$}
\label{table6}
\end{table}

\begin{table}[htb]
\label{tab:t7}
\begin{center}
\renewcommand{\arraystretch}{1.3}
\begin{tabular}{|c|p{2.50in}|l|}\hline
$k$ & $f_k$ & $\Delta_k(t)$  \\\hline
$4$ & $\omega=e_{1234567}$ & $(t-i)^8$  \\\hline
$3$ & $\frac14(e_{4}+e_{145}+e_{246}+e_{347}-e_{12456}\newline \hspace*{20ex}-e_{13457}-e_{23467}+3\omega)$ & $(t-i)^7(t+i)$  \\\hline
$2$ & $\frac12(e_{145}-e_{12456}-e_{13457}+\omega)$ & $(t-i)^6(t+i)^2$  \\\hline
$1$ & $\frac14(-e_{4}+e_{145}+e_{246}-e_{347}-3e_{12456}\newline \hspace*{20ex}-e_{13457}-e_{23467}+\omega)$ & $(t-i)^5(t+i)^3$  \\\hline
$0$ & $\frac12(e_{4}+e_{12456}+e_{13457}+e_{23467})$ & $(t-i)^4(t+i)^4$   \\\hline
$-1$ & $\frac14(-e_{4}-e_{145}-e_{246}+e_{347}-3e_{12456}\newline \hspace*{20ex}-e_{13457}-e_{23467}-\omega)$ & $(t-i)^3(t+i)^5$ \\\hline
$-2$ & $\frac12(-e_{145}-e_{12456}-e_{13457}-\omega)$ & $(t-i)^2(t+i)^6$   \\\hline
$-3$ & $\frac14(e_{4}-e_{145}-e_{246}-e_{347}-e_{12456}\newline \hspace*{20ex}-e_{13457}-e_{23467}-3\omega)$ & $(t-i)(t+i)^7$   \\\hline
$-4$ & $-\omega=-e_{1234567}$ & $(t+i)^8$   \\\hline
\end{tabular}
\end{center}
\caption{Square roots of $-\unitm$ in $\cl(3,4)\cong \M(8,\C)$, $d=4$}
\label{table7}
\end{table}

{\small
\begin{table}[htb]
\label{tab:t8}
\begin{center}
\renewcommand{\arraystretch}{1.3}
\begin{tabular}{|c|p{2.50in}|l|}\hline
$k$ & $f_k$ & $\Delta_k(t)$  \\\hline
$4$ & $\omega=e_{1234567}$ & $(t-i)^8$  \\\hline
$3$ & $\frac14(-e_{23}+e_{123}+e_{2346}+e_{2357}-e_{12346}\newline \hspace*{20ex}-e_{12357}+e_{234567}+3\omega)$ & $(t-i)^7(t+i)$  \\\hline
$2$ & $\frac12(e_{123}-e_{12346}-e_{12357}+\omega)$ & $(t-i)^6(t+i)^2$  \\\hline
$1$ & $\frac14(-e_{23}+e_{123}-e_{2346}+e_{2357}-3e_{12346}\newline \hspace*{20ex}-e_{12357}-e_{234567}+\omega)$ & $(t-i)^5(t+i)^3$  \\\hline
$0$ & $\frac12(e_{23}+e_{12346}+e_{12357}+e_{234567})$ & $(t-i)^4(t+i)^4$   \\\hline
$-1$ & $\frac14(-e_{23}-e_{123}+e_{2346}-e_{2357}-3e_{12346}\newline \hspace*{20ex}-e_{12357}-e_{234567}-\omega)$ & $(t-i)^3(t+i)^5$ \\\hline
$-2$ & $\frac12(-e_{123}-e_{12346}-e_{12357}-\omega)$ & $(t-i)^2(t+i)^6$   \\\hline
$-3$ & $\frac14(-e_{23}-e_{123}-e_{2346}-e_{2357}-e_{12346}\newline \hspace*{20ex}-e_{12357}+e_{234567}-3\omega)$ & $(t-i)(t+i)^7$   \\\hline
$-4$ & $-\omega=-e_{1234567}$ & $(t+i)^8$   \\\hline
\end{tabular}
\end{center}
\caption{Square roots of $-\unitm$ in $\cl(5,2)\cong \M(8,\C)$, $d=4$}
\label{table8}
\end{table}
}

\newpage
\section{A sample Maple worksheet}
\label{AppendB}
In this appendix we show a computation of roots of $-\unitm$ in $\cl(3,0)$ in CLIFFORD. Although these computations certainly can be performed by hand, as shown in Section \ref{sc:M2dC}, they illustrate how CLIFFORD can be used instead especially when extending these computations to higher dimensions.\footnote{In showing Maple display we have edited Maple output to save space. Package $\mathtt{asvd}$ is a supplementary package written by the third author and built into CLIFFORD. The primary purpose of $\mathtt{asvd}$ is to compute Singular Value Decomposition in Clifford algebras~\cite{asvd}.} To see the actual Maple worksheets where these computations have been performed, see~\cite{worksheets}. 
\begin{smallmaplegroup}
\begin{mapleinput}
\mapleinline{active}{1d}{restart:with(Clifford):with(linalg):with(asvd):}{%
}
\mapleinline{active}{1d}{p,q:=3,0; ##<<-- selecting signature}{%
}
\mapleinline{active}{1d}{B:=diag(1$p,-1$q): ##<<-- defining diagonal bilinear form}{%
}
\mapleinline{active}{1d}{eval(makealiases(p+q)): ##<<-- defining aliases}{%
}
\mapleinline{active}{1d}{clibas:=cbasis(p+q); ##assigning basis for Cl(3,0)}{%
}
\end{mapleinput}

\mapleresult
\begin{maplelatex}
\mapleinline{inert}{2d}{p, q := 3, 0;}{%
\[
p, \,q := 3, \,0
\]
}
\mapleinline{inert}{2d}{clibas := [Id, e1, e2, e3, e12, e13, e23, e123];}{%
\[
\mathit{clibas} := [\mathit{Id}, \,\mathit{e1}, \,\mathit{e2}, \,
\mathit{e3}, \,\mathit{e12}, \,\mathit{e13}, \,\mathit{e23}, \,
\mathit{e123}]
\]
}
\end{maplelatex}
\end{smallmaplegroup}
\begin{smallmaplegroup}
\begin{mapleinput}
\mapleinline{active}{1d}{data:=clidata(); ##<<-- displaying information about Cl(3,0)}{%
}
\end{mapleinput}
\mapleresult
\begin{maplelatex}
\mapleinline{inert}{2d}{data := [complex, 2, simple, 1/2*Id+1/2*e1, [Id, e2, e3, e23], [Id,
e23], [Id, e2]];}{%
\[
\mathit{data} := [\mathit{complex}, \,2, \,\mathit{simple}, \,
{\displaystyle \frac {\mathit{Id}}{2}}  + {\displaystyle \frac {
\mathit{e1}}{2}} , \,[\mathit{Id}, \,\mathit{e2}, \,\mathit{e3}, 
\,\mathit{e23}], \,[\mathit{Id}, \,\mathit{e23}], \,[\mathit{Id}
, \,\mathit{e2}]]
\]
}
\end{maplelatex}
\end{smallmaplegroup}
\begin{smallmaplegroup}
\begin{mapleinput}
\mapleinline{active}{1d}{MM:=matKrepr(); ##<<-- displaying default matrices to generators}{%
}
\end{mapleinput}
\mapleresult
\begin{maplettyout}
Cliplus has been loaded. Definitions for type/climon and
type/clipolynom now include &C and &C[K]. Type ?cliprod for help.
\end{maplettyout}

\begin{maplelatex}
\mapleinline{inert}{2d}{MM := [e1 = matrix([[1, 0], [0, -1]]), e2 = matrix([[0, 1], [1, 0]]),
e3 = matrix([[0, -e23], [e23, 0]])];}{%
\[
\mathit{MM} := [\mathit{e1}= \left[ 
{\begin{array}{rr}
1 & 0 \\
0 & -1
\end{array}}
 \right] , \,\mathit{e2}= \left[ 
{\begin{array}{rr}
0 & 1 \\
1 & 0
\end{array}}
 \right] , \,\mathit{e3}= \left[ 
{\begin{array}{cc}
0 &  - \mathit{e23} \\
\mathit{e23} & 0
\end{array}}
 \right] ]
\]
}
\end{maplelatex}
\end{smallmaplegroup}
\noindent
Pauli algebra representation displayed in (7.1):
\begin{smallmaplegroup}
\begin{mapleinput}
\mapleinline{active}{1d}{sigma[1]:=evalm(rhs(MM[1]));}{%
}
\mapleinline{active}{1d}{sigma[2]:=evalm(rhs(MM[2]));}{%
}
\mapleinline{active}{1d}{sigma[3]:=evalm(rhs(MM[3]));}{%
}
\end{mapleinput}
\mapleresult
\begin{maplelatex}
\mapleinline{inert}{2d}{sigma[1], sigma[2], sigma[3] := matrix([[0, 1], [1, 0]]), matrix([[0,
-e23], [e23, 0]]), matrix([[1, 0], [0, -1]]);}{%
\[
{\sigma _{1}}, \,{\sigma _{2}}, \,{\sigma _{3}} :=  \left[ 
{\begin{array}{rr}
0 & 1 \\
1 & 0
\end{array}}
 \right] , \, \left[ 
{\begin{array}{cc}
0 &  - \mathit{e23} \\
\mathit{e23} & 0
\end{array}}
 \right] , \, \left[ 
{\begin{array}{rr}
1 & 0 \\
0 & -1
\end{array}}
 \right] 
\]
}
\end{maplelatex}
\end{smallmaplegroup}
We show how we represent the imaginary unit $i$ in the field $\C$ and the diagonal matrix 
$\mathrm{diag}(i,i):$
\begin{smallmaplegroup}
\begin{mapleinput}
\mapleinline{active}{1d}{ii:=e23;         ##<<-- complex imaginary unit}{%
}
\mapleinline{active}{1d}{II:=diag(ii,ii); ##<<-- diagonal matrix diag(i,i) }{%
}
\end{mapleinput}
\mapleresult
\begin{maplelatex}
\mapleinline{inert}{2d}{ii := e23;}{%
\[
\mathit{ii} := \mathit{e23}
\]
}
\mapleinline{inert}{2d}{II := matrix([[e23, 0], [0, e23]]);}{%
\[
\mathit{II} :=  \left[ 
{\begin{array}{cc}
\mathit{e23} & 0 \\
0 & \mathit{e23}
\end{array}}
 \right] 
\]
}
\end{maplelatex}
\end{smallmaplegroup}

\noindent
We compute matrices $m_1, m_2, \ldots, m_8$ representing each basis element in 
$\cl(3,0)$ isomorphic with $\C(2)$. Note that in our representation element 
$e_{23}$ in $\cl(3,0)$ is used to represent the imaginary unit $i$.
\begin{smallmaplegroup}
\begin{mapleinput}
\mapleinline{active}{1d}{for i from 1 to nops(clibas) do\newline 
\mytab lprint(`The basis element`,clibas[i],`is represented by the following\newline
\mytab matrix:`);\newline
\mytab M[i]:=subs(Id=1,matKrepr(clibas[i])) od;}{%
}
\end{mapleinput}

\mapleresult
\begin{maplettyout}
`The basis element`, Id, `is represented by the following matrix:`
\end{maplettyout}
\begin{maplelatex}
\mapleinline{inert}{2d}{M[1] := matrix([[1, 0], [0, 1]]);}{%
\[
{M_{1}} :=  \left[ 
{\begin{array}{rr}
1 & 0 \\
0 & 1
\end{array}}
 \right] 
\]
}
\end{maplelatex}
\begin{maplettyout}
`The basis element`, e1, `is represented by the following matrix:`
\end{maplettyout}

\begin{maplelatex}
\mapleinline{inert}{2d}{M[2] := matrix([[1, 0], [0, -1]]);}{%
\[
{M_{2}} :=  \left[ 
{\begin{array}{rr}
1 & 0 \\
0 & -1
\end{array}}
 \right] 
\]
}
\end{maplelatex}

\begin{maplettyout}
`The basis element`, e2, `is represented by the following matrix:`
\end{maplettyout}

\begin{maplelatex}
\mapleinline{inert}{2d}{M[3] := matrix([[0, 1], [1, 0]]);}{%
\[
{M_{3}} :=  \left[ 
{\begin{array}{rr}
0 & 1 \\
1 & 0
\end{array}}
 \right] 
\]
}
\end{maplelatex}

\begin{maplettyout}
`The basis element`, e3, `is represented by the following matrix:`
\end{maplettyout}

\begin{maplelatex}
\mapleinline{inert}{2d}{M[4] := matrix([[0, -e23], [e23, 0]]);}{%
\[
{M_{4}} :=  \left[ 
{\begin{array}{cc}
0 &  - \mathit{e23} \\
\mathit{e23} & 0
\end{array}}
 \right] 
\]
}
\end{maplelatex}

\begin{maplettyout}
`The basis element`, e12, `is represented by the following matrix:`
\end{maplettyout}

\begin{maplelatex}
\mapleinline{inert}{2d}{M[5] := matrix([[0, 1], [-1, 0]]);}{%
\[
{M_{5}} :=  \left[ 
{\begin{array}{rr}
0 & 1 \\
-1 & 0
\end{array}}
 \right] 
\]
}
\end{maplelatex}

\begin{maplettyout}
`The basis element`, e13, `is represented by the following matrix:`
\end{maplettyout}

\begin{maplelatex}
\mapleinline{inert}{2d}{M[6] := matrix([[0, -e23], [-e23, 0]]);}{%
\[
{M_{6}} :=  \left[ 
{\begin{array}{cc}
0 &  - \mathit{e23} \\
 - \mathit{e23} & 0
\end{array}}
 \right] 
\]
}
\end{maplelatex}

\begin{maplettyout}
`The basis element`, e23, `is represented by the following matrix:`
\end{maplettyout}

\begin{maplelatex}
\mapleinline{inert}{2d}{M[7] := matrix([[e23, 0], [0, -e23]]);}{%
\[
{M_{7}} :=  \left[ 
{\begin{array}{cc}
\mathit{e23} & 0 \\
0 &  - \mathit{e23}
\end{array}}
 \right] 
\]
}
\end{maplelatex}

\begin{maplettyout}
`The basis element`, e123, `is represented by the following matrix:`
\end{maplettyout}

\begin{maplelatex}
\mapleinline{inert}{2d}{M[8] := matrix([[e23, 0], [0, e23]]);}{%
\[
{M_{8}} :=  \left[ 
{\begin{array}{cc}
\mathit{e23} & 0 \\
0 & \mathit{e23}
\end{array}}
 \right] 
\]
}
\end{maplelatex}
\end{smallmaplegroup}

We will use the procedure $\mathtt{phi}$ from the $\mathtt{asvd}$ package which gives an
isomorphism from $\C(2)$ to $\cl(3,0)$.  This way we can find the image
in $\cl(3,0)$ of any complex $2 \times 2$ complex matrix $A$. Knowing the
image of each matrix $m_1, m_2, \ldots, m_8$ in terms of the Clifford
polynomials in $\cl(3,0)$, we can easily find the image of $A$ in our default spinor representation of $\cl(3,0)$ which is built into CLIFFORD.

Procedure $\mathtt{Centralizer}$ computes a centralizer of $f$ with respect to the
Clifford basis $L$:
\begin{smallmaplegroup}
\begin{mapleinput}
\mapleinline{active}{1d}{Centralizer:=proc(f,L) local c,LL,m,vars,i,eq,sol;\newline
\mytab m:=add(c[i]*L[i],i=1..nops(L));\newline
\mytab vars:=[seq(c[i],i=1..nops(L))];\newline
\mytab eq:=clicollect(cmul(f,m)-cmul(m,f));\newline
\mytab if eq=0 then return L end if:\newline
\mytab sol:=op(clisolve(eq,vars));\newline
\mytab m:=subs(sol,m);\newline
\mytab m:=collect(m,vars);\newline
\mytab return sort([coeffs(m,vars)],bygrade);\newline
\mytab end proc:}{%
}
\end{mapleinput}
\end{smallmaplegroup}

Procedures $\mathtt{Scal}$ and $\mathtt{Spec}$ compute the scalar and the pseudoscalar parts
of $f$.
\begin{smallmaplegroup}
\begin{mapleinput}
\mapleinline{active}{1d}{Scal:=proc(f) local p:
return scalarpart(f);
end proc:}{%
}
\mapleinline{active}{1d}{Spec:=proc(f) local N; global p,q;\newline
\mytab N:=p+q:\newline
\mytab return coeff(vectorpart(f,N),op(cbasis(N,N)));\newline
\mytab end proc:}{%
}
\end{mapleinput}
\end{smallmaplegroup}
\noindent
The matrix idempotents in $\C(2)$ displayed in (7.2) are as follows:
\begin{smallmaplegroup}
\begin{mapleinput}
\mapleinline{active}{1d}{d:=1:Eps[1]:=sigma[1] &cm sigma[1];}{%
}
\mapleinline{active}{1d}{Eps[0]:=evalm(1/2*(1+sigma[3]));}{%
}
\mapleinline{active}{1d}{Eps[-1]:=diag(0,0);}{%
}
\end{mapleinput}

\mapleresult
\begin{maplelatex}
\mapleinline{inert}{2d}{Eps[1], Eps[0], Eps[-1] := matrix([[1, 0], [0, 1]]), matrix([[1, 0],
[0, 0]]), matrix([[0, 0], [0, 0]]);}{%
\[
{\mathit{Eps}_{1}}, \,{\mathit{Eps}_{0}}, \,{\mathit{Eps}_{-1}}
 :=  \left[ 
{\begin{array}{rr}
1 & 0 \\
0 & 1
\end{array}}
 \right] , \, \left[ 
{\begin{array}{rr}
1 & 0 \\
0 & 0
\end{array}}
 \right] , \, \left[ 
{\begin{array}{rr}
0 & 0 \\
0 & 0
\end{array}}
 \right] 
\]
}
\end{maplelatex}
\end{smallmaplegroup}
This function $\mathtt{ff}$ computes matrix square root of $-1$ corresponding to
the matrix idempotent $eps$:
\begin{smallmaplegroup}
\begin{mapleinput}
\mapleinline{active}{1d}{ff:=eps->evalm(II &cm (2*eps-1));}{%
}
\end{mapleinput}

\mapleresult
\begin{maplelatex}
\mapleinline{inert}{2d}{ff := eps -> evalm(`&cm`(II,2*eps-1));}{%
\[
\mathit{ff} := \mathit{eps}\rightarrow \mathrm{evalm}(\mathit{II}
\,\mathrm{\&cm}\,(2\,\mathit{eps} - 1))
\]
}
\end{maplelatex}
\end{smallmaplegroup}
\noindent
We compute matrix square roots of $-1$ which correspond to the
idempotents $Eps_1, Eps_0, Eps_{-1}$, and their characteristic and
minimal polynomials. Note that in Maple the default imaginary unit is
denoted by $I$.
\begin{smallmaplegroup}
\begin{mapleinput}
\mapleinline{active}{1d}{F[1]:=ff(Eps[1]); ##<<--this square root of -1 corresponds to Eps[1]\newline
\mytab Delta[1]:=charpoly(subs(e23=I,evalm(F[1])),t);\newline
\mytab Mu[1]:=minpoly(subs(e23=I,evalm(F[1])),t);}{%
}
\end{mapleinput}

\mapleresult
\begin{maplelatex}
\mapleinline{inert}{2d}{F[1] := matrix([[e23, 0], [0, e23]]);}{%
\[
{F_{1}} :=  \left[ 
{\begin{array}{cc}
\mathit{e23} & 0 \\
0 & \mathit{e23}
\end{array}}
 \right], 
\quad
\Delta _{1} := (t - I)^{2},
\quad  
M_{1} := t - I
\]
}
\end{maplelatex}
\end{smallmaplegroup}
\begin{smallmaplegroup}
\begin{mapleinput}
\mapleinline{active}{1d}{F[0]:=ff(Eps[0]); ##<<--this square root of -1 corresponds to Eps[0]\newline
\mytab Delta[0]:=charpoly(subs(e23=I,evalm(F[0])),t);\newline
\mytab Mu[0]:=minpoly(subs(e23=I,evalm(F[0])),t);}{%
}
\end{mapleinput}

\mapleresult
\begin{maplelatex}
\mapleinline{inert}{2d}{F[0] := matrix([[e23, 0], [0, -e23]]);}{%
\[
{F_{0}} :=  \left[ 
{\begin{array}{cc}
\mathit{e23} & 0 \\
0 &  - \mathit{e23}
\end{array}}
 \right],
\quad
{\Delta _{0}} := (t - I)\,(t + I),
\quad
{M_{0}} := 1 + t^{2}
\]
}
\end{maplelatex}

\end{smallmaplegroup}
\begin{smallmaplegroup}
\begin{mapleinput}
\mapleinline{active}{1d}{F[-1]:=ff(Eps[-1]); ##<<--this square root of -1 corresponds to
Eps[-1]\newline
\mytab Delta[-1]:=charpoly(subs(e23=I,evalm(F[-1])),t);\newline
\mytab Mu[-1]:=minpoly(subs(e23=I,evalm(F[-1])),t);}{%
}
\end{mapleinput}

\mapleresult
\begin{maplelatex}
\mapleinline{inert}{2d}{F[-1] := matrix([[-e23, 0], [0, -e23]]);}{%
\[
{F_{-1}} :=  \left[ 
{\begin{array}{cc}
 - \mathit{e23} & 0 \\
0 &  - \mathit{e23}
\end{array}}
 \right],
\quad 
{\Delta _{-1}} := (t + I)^{2},
\quad
{M_{-1}} := t + I
\]
}
\end{maplelatex}
\end{smallmaplegroup}
\noindent
Now, we can find square roots of $-1$ in $\cl(3,0)$ which correspond to the
matrix square roots $F_{-1}, F_{0}, F_{1}$ via the isomorphism 
$\phi: \cl(3,0) \rightarrow \C(2)$ realized with the procedure $\mathtt{phi}$. 

First, we let $\mathtt{reprI}$ denote element in $\cl(3,0)$ which represents the
diagonal $(2d) \times (2d)$ with $I=i$ on the diagonal where $i^2 = -1$. This element will replace the imaginary unit $I$ in the minimal polynomials. 
\begin{smallmaplegroup}
\begin{mapleinput}
\mapleinline{active}{1d}{reprI:=phi(diag(I$(2*d)),M);}{%
}
\end{mapleinput}

\mapleresult
\begin{maplelatex}
\mapleinline{inert}{2d}{reprI := e123;}{%
\[
\mathit{reprI} := \mathit{e123}
\]
}
\end{maplelatex}
\end{smallmaplegroup}
\noindent
Now, we compute the corresponding square roots $f_1,f_0,f_{-1}$ in $\cl(3,0).$

\begin{smallmaplegroup}
\begin{mapleinput}
\mapleinline{active}{1d}{f[1]:=phi(F[1],M); ##<<--element in Cl(3,0) corresponding to F[1]\newline
\mytab cmul(f[1],f[1]); ##<<--checking that this element is a root of -1\newline
\mytab Mu[1]; ##<<--recalling minpoly of matrix F[1]\newline
\mytab subs(e23=I,evalm(subs(t=evalm(F[1]),Mu[1]))); ##<<--F[1] in Mu[1]\newline 
\mytab mu[1]:=subs(I=reprI,Mu[1]); ##<<--defining minpoly of f[1]\newline
\mytab cmul(f[1]-reprI,Id); ##<<-- verifying that f[1] satisfies mu[1]}{%
}
\end{mapleinput}

\mapleresult
\begin{maplelatex}
\mapleinline{inert}{2d}{f[1] := e123;}{%
\[
{f_{1}} := \mathit{e123}
\]
}
\mapleinline{inert}{2d}{-Id;}{%
\[
 - \mathit{Id},
\quad
t - I,
\quad
 \left[ 
{\begin{array}{rr}
0 & 0 \\
0 & 0
\end{array}}
 \right] 
\]
}
\mapleinline{inert}{2d}{mu[1] := t-e123;}{%
\[
{\mu _{1}} := t - \mathit{e123},
\quad
0
\]
}
\end{maplelatex}
\end{smallmaplegroup}
\begin{smallmaplegroup}
\begin{mapleinput}
\mapleinline{active}{1d}{f[0]:=phi(F[0],M); ##<<--element in Cl(3,0) corresponding to F[0]\newline
\mytab cmul(f[0],f[0]); ##<<-- checking that this element is a root of -1\newline
\mytab Mu[0]; ##<<--recalling minpoly of matrix F[0]\newline
\mytab subs(e23=I,evalm(subs(t=evalm(F[0]),Mu[0]))); ##<<--F[0] in Mu[0]\newline
\mytab mu[0]:=subs(I=reprI,Mu[0]); ##<<--defining minpoly of f[0]\newline
\mytab cmul(f[0]-reprI,f[0]+reprI); ##<<--f[0] satisfies mu[0]}{%
}
\end{mapleinput}
\mapleresult
\begin{maplelatex}
\mapleinline{inert}{2d}{f[0] := e23;}{%
\[
{f_{0}} := \mathit{e23}
\]
}
\mapleinline{inert}{2d}{-Id;}{%
\[
 - \mathit{Id},
\quad
1 + t^{2},
\quad
 \left[ 
{\begin{array}{rr}
0 & 0 \\
0 & 0
\end{array}}
 \right] 
\]
}
\mapleinline{inert}{2d}{mu[0] := 1+t^2;}{%
\[
{\mu _{0}} := 1 + t^{2},
\quad
0
\]
}
\end{maplelatex}
\end{smallmaplegroup}
\begin{smallmaplegroup}
\begin{mapleinput}
\mapleinline{active}{1d}{f[-1]:=phi(F[-1],M); ##<<--element in Cl(3,0) corresponding to F[-1]\newline
\mytab cmul(f[-1],f[-1]);  ##<<--checking that this element is a root of -1\newline
\mytab Mu[-1]; ##<<--recalling minpoly of matrix F[-1]\newline
\mytab subs(e23=I,evalm(subs(t=evalm(F[-1]),Mu[-1]))); ##<<--F[-1] in Mu[-1]\newline
\mytab mu[-1]:=subs(I=reprI,Mu[-1]); ##<<--defining minpoly of f[-1]\newline
\mytab cmul(f[-1]+reprI,Id); ##<<--f[-1] satisfies mu[-1]}{%
}
\end{mapleinput}
\mapleresult
\begin{maplelatex}
\mapleinline{inert}{2d}{f[-1] := -e123;}{%
\[
{f_{-1}} :=  - \mathit{e123}
\]
}
\mapleinline{inert}{2d}{-Id;}{%
\[
 - \mathit{Id},
\quad
t + I,
\quad
 \left[ 
{\begin{array}{rr}
0 & 0 \\
0 & 0
\end{array}}
 \right] 
\]
}
\mapleinline{inert}{2d}{mu[-1] := t+e123;}{%
\[
{\mu _{-1}} := t + \mathit{e123}, 
\quad
0
\]
}
\end{maplelatex}
\end{smallmaplegroup}

Functions $\mathtt{RdimCentralizer}$ and $\mathtt{RdimConjugClass}$ of $d$ and $k$ compute the
real dimension of the centralizer $\mathrm{Cent}(f)$ and the conjugacy class of $f$ (see (7.4)).
\begin{smallmaplegroup}
\begin{mapleinput}
\mapleinline{active}{1d}{RdimCentralizer:=(d,k)->2*((d+k)^2+(d-k)^2); ##<<--from the theory}{%
}
\mapleinline{active}{1d}{RdimConjugClass:=(d,k)->4*(d^2-k^2); ##<<--from the theory}{%
}
\end{mapleinput}

\mapleresult
\begin{maplelatex}
\mapleinline{inert}{2d}{RdimCentralizer := (d, k) -> 2*(d+k)^2+2*(d-k)^2;}{%
\[
\mathit{RdimCentralizer} := (d, \,k)\rightarrow 2\,(d + k)^{2} + 
2\,(d - k)^{2}
\]
}
\mapleinline{inert}{2d}{RdimConjugClass := (d, k) -> 4*d^2-4*k^2;}{%
\[
\mathit{RdimConjugClass} := (d, \,k)\rightarrow 4\,d^{2} - 4\,k^{2}
\]
}
\end{maplelatex}
\end{smallmaplegroup}
Now, we compute the centralizers of the roots and use notation $d, k, n_1, n_2$ displayed in Examples.

\noindent
Case $k=1:$

\begin{smallmaplegroup}
\begin{mapleinput}
\mapleinline{active}{1d}{d:=1:k:=1:n1:=d+k;n2:=d-k;\newline
\mytab A1:=diag(I$n1,-I$n2); ##<<-- this is the first matrix root of -1}{%
}
\end{mapleinput}

\mapleresult
\begin{maplelatex}
\mapleinline{inert}{2d}{n1 := 2;}{%
\[
\mathit{n1} := 2, 
\quad
\mathit{n2} := 0, 
\quad 
\mathit{A1} :=  \left[ 
{\begin{array}{cc}
I & 0 \\
0 & I
\end{array}}
 \right] 
\]
}
\end{maplelatex}
\end{smallmaplegroup}
\begin{smallmaplegroup}
\begin{mapleinput}
\mapleinline{active}{1d}{f[1]:=phi(A1,M); cmul(f[1],f[1]); Scal(f[1]), Spec(f[1]);}{%
}
\end{mapleinput}

\mapleresult
\begin{maplelatex}
\mapleinline{inert}{2d}{f[1] := e123;}{%
\[
{f_{1}} := \mathit{e123}, 
\quad
- \mathit{Id},
\quad
0, \quad 1
\]
}
\end{maplelatex}
\end{smallmaplegroup}
\begin{smallmaplegroup}
\begin{mapleinput}
\mapleinline{active}{1d}{LL1:=Centralizer(f[1],clibas); ##<<--centralizer of f[1]\newline
\mytab dimCentralizer:=nops(LL1); ##<<--real dimension of centralizer of f[1]\newline
\mytab RdimCentralizer(d,k); ##<<--dimension of centralizer of f[1] from theory\newline
\mytab evalb(dimCentralizer=RdimCentralizer(d,k)); ##<<--checking equality
}{%
}
\end{mapleinput}
\mapleresult
\begin{maplelatex}
\mapleinline{inert}{2d}{LL1 := [Id, e1, e2, e3, e12, e13, e23, e123];}{%
\[
\mathit{LL1} := [\mathit{Id}, \,\mathit{e1}, \,\mathit{e2}, \,
\mathit{e3}, \,\mathit{e12}, \,\mathit{e13}, \,\mathit{e23}, \,
\mathit{e123}]
\]
}
\mapleinline{inert}{2d}{dimCentralizer := 8;}{%
\[
\mathit{dimCentralizer} := 8,
\quad
8,
\quad
\mathit{true}
\]
}
\end{maplelatex}
\end{smallmaplegroup}
\noindent
Case $k=0:$

\begin{smallmaplegroup}
\begin{mapleinput}
\mapleinline{active}{1d}{d:=1:k:=0:n1:=d+k;n2:=d-k;\newline
\mytab A0:=diag(I$n1,-I$n2);  ##<<-- this is the second matrix root of -1}{%
}
\end{mapleinput}

\mapleresult
\begin{maplelatex}
\mapleinline{inert}{2d}{n1 := 1;}{%
\[
\mathit{n1} := 1,
\quad
\mathit{n2} := 1,
\quad
\mathit{A0} :=  \left[ 
{\begin{array}{cc}
I & 0 \\
0 &  - I
\end{array}}
 \right] 
\]
}
\end{maplelatex}
\end{smallmaplegroup}
\begin{smallmaplegroup}
\begin{mapleinput}
\mapleinline{active}{1d}{f[0]:=phi(A0,M); cmul(f[0],f[0]); Scal(f[0]), Spec(f[0]);}{%
}
\end{mapleinput}

\mapleresult
\begin{maplelatex}
\mapleinline{inert}{2d}{f[0] := e23;}{%
\[
{f_{0}} := \mathit{e23}, 
\quad
- \mathit{Id},
\quad
0, \quad 0
\]
}
\end{maplelatex}

\end{smallmaplegroup}
\begin{smallmaplegroup}
\begin{mapleinput}
\mapleinline{active}{1d}{LL0:=Centralizer(f[0],clibas); ##<<--centralizer of f[0]\newline
\mytab dimCentralizer:=nops(LL0); ##<<--real dimension of centralizer of f[0]\newline
\mytab RdimCentralizer(d,k); ##<<--dimension of centralizer of f[0] from theory\newline
\mytab evalb(dimCentralizer=RdimCentralizer(d,k)); ##<<--checking equality
}{%
}
\end{mapleinput}

\mapleresult
\begin{maplelatex}
\mapleinline{inert}{2d}{LL0 := [Id, e1, e23, e123];}{%
\[
\mathit{LL0} := [\mathit{Id}, \,\mathit{e1}, \,\mathit{e23}, \,
\mathit{e123}]
\]
}
\mapleinline{inert}{2d}{dimCentralizer := 4;}{%
\[
\mathit{dimCentralizer} := 4,
\quad
4,
\quad
\mathit{true}
\]
}
\end{maplelatex}
\end{smallmaplegroup}
\noindent
Case $k=-1:$

\begin{smallmaplegroup}
\begin{mapleinput}
\mapleinline{active}{1d}{d:=1:k:=-1:n1:=d+k;n2:=d-k; \newline
\mytab Am1:=diag(I$n1,-I$n2); ##<<-- this is the third matrix root of -1}{%
}
\end{mapleinput}

\mapleresult
\begin{maplelatex}
\mapleinline{inert}{2d}{n1 := 0;}{%
\[
\mathit{n1} := 0,
\quad
\mathit{n2} := 2,
\quad
\mathit{Am1} :=  \left[ 
{\begin{array}{cc}
 - I & 0 \\
0 &  - I
\end{array}}
 \right] 
\]
}
\end{maplelatex}
\end{smallmaplegroup}

\begin{smallmaplegroup}
\begin{mapleinput}
\mapleinline{active}{1d}{f[-1]:=phi(Am1,M); cmul(f[-1],f[-1]); Scal(f[-1]), Spec(f[-1]);}{%
}
\end{mapleinput}

\mapleresult
\begin{maplelatex}
\mapleinline{inert}{2d}{f[-1] := -e123;}{%
\[
{f_{-1}} :=  - \mathit{e123},
\quad
 - \mathit{Id},
\quad
0, \quad -1
\]
}
\end{maplelatex}
\end{smallmaplegroup}

\begin{smallmaplegroup}
\begin{mapleinput}
\mapleinline{active}{1d}{LLm1:=Centralizer(f[-1],clibas); ##<<--centralizer of f[-1]\newline
\mytab dimCentralizer:=nops(LLm1); ##<<--real dimension of centralizer of f[-1]\newline
\mytab RdimCentralizer(d,k); ##<<--dimension of centralizer of f[-1] from theory
\mytab evalb(dimCentralizer=RdimCentralizer(d,k)); ##<<--checking equality
}{%
}
\end{mapleinput}

\mapleresult
\begin{maplelatex}
\mapleinline{inert}{2d}{LLm1 := [Id, e1, e2, e3, e12, e13, e23, e123];}{%
\[
\mathit{LLm1} := [\mathit{Id}, \,\mathit{e1}, \,\mathit{e2}, \,
\mathit{e3}, \,\mathit{e12}, \,\mathit{e13}, \,\mathit{e23}, \,
\mathit{e123}]
\]
}
\mapleinline{inert}{2d}{dimCentralizer := 8;}{%
\[
\mathit{dimCentralizer} := 8,
\quad
8,
\quad
\mathit{true}
\]
}
\end{maplelatex}
\end{smallmaplegroup}

We summarize roots of $-1$ in $\cl(3,0)$:
\begin{smallmaplegroup}
\begin{mapleinput}
\mapleinline{active}{1d}{'F[1]'=evalm(F[1]); ##<<--square root of -1 in C(2)\newline
\mytab Mu[1]; ##<<--minpoly of matrix F[1]\newline
\mytab 'f[1]'=f[1]; ##<<--square root of -1 in Cl(3,0)\newline
\mytab mu[1]; ##<<--minpoly of element f[1] 
}{%
}
\end{mapleinput}
\mapleresult
\begin{maplelatex}
\mapleinline{inert}{2d}{F[1] = matrix([[e23, 0], [0, e23]]);}{%
\[
{F_{1}}= \left[ 
{\begin{array}{cc}
\mathit{e23} & 0 \\
0 & \mathit{e23}
\end{array}}
 \right], 
\quad
t - I
\]
}
\mapleinline{inert}{2d}{f[1] = e123;}{%
\[
{f_{1}}=\mathit{e123},
\quad
t - \mathit{e123}
\]
}
\end{maplelatex}
\end{smallmaplegroup}
\begin{smallmaplegroup}
\begin{mapleinput}
\mapleinline{active}{1d}{'F[0]'=evalm(F[0]); ##<<-- square root of -1 in C(2)\newline
\mytab Mu[0]; ##<<--minpoly of matrix F[0]\newline
\mytab 'f[0]'=f[0]; ##<<--square root of -1 in Cl(3,0)\newline
\mytab mu[0]; ##<<--minpoly of element f[0] 
}{%
}
\end{mapleinput}

\mapleresult
\begin{maplelatex}
\mapleinline{inert}{2d}{F[0] = matrix([[e23, 0], [0, -e23]]);}{%
\[
{F_{0}}= \left[ 
{\begin{array}{cc}
\mathit{e23} & 0 \\
0 &  - \mathit{e23}
\end{array}}
 \right],
\quad
1 + t^{2}
\]
}
\end{maplelatex}

\begin{maplelatex}
\mapleinline{inert}{2d}{f[0] = e23;}{%
\[
{f_{0}}=\mathit{e23},
\quad
1 + t^{2}
\]
}
\end{maplelatex}
\end{smallmaplegroup}
\begin{smallmaplegroup}
\begin{mapleinput}
\mapleinline{active}{1d}{'F[-1]'=evalm(F[-1]); ##<<--square root of -1 in C(2)\newline
\mytab Mu[-1]; ##<<--minpoly of matrix F[-1]\newline
\mytab 'f[-1]'=f[-1]; ##<<--square root of -1 in Cl(3,0)\newline
\mytab mu[-1]; ##<<--minpoly of element f[-1] 
}{%
}
\end{mapleinput}

\mapleresult
\begin{maplelatex}
\mapleinline{inert}{2d}{F[-1] = matrix([[-e23, 0], [0, -e23]]);}{%
\[
{F_{-1}}= \left[ 
{\begin{array}{cc}
 - \mathit{e23} & 0 \\
0 &  - \mathit{e23}
\end{array}}
 \right],
\quad
t + I
\]
}
\mapleinline{inert}{2d}{f[-1] = -e123;}{%
\[
{f_{-1}}= - \mathit{e123},
\quad
t + \mathit{e123}
\]
}
\end{maplelatex}
\end{smallmaplegroup}
Finaly, we verify that roots $f_1$ and $f_{-1}$ are related via the reversion:

\begin{smallmaplegroup}
\begin{mapleinput}
\mapleinline{active}{1d}{reversion(f[1])=f[-1]; evalb(\%);}{%
}
\end{mapleinput}

\mapleresult
\begin{maplelatex}
\mapleinline{inert}{2d}{-e123 = -e123;}{%
\[
 - \mathit{e123}= - \mathit{e123},
\quad
\mathit{true}
\]
}
\end{maplelatex}
\end{smallmaplegroup}


\end{document}